\documentclass{amsart}
\usepackage{amsfonts,  amssymb, bbm}

\usepackage{overpic}
\newtheorem{theorem}{Theorem}[section]
\theoremstyle{plain}
\newtheorem{corollary}[theorem]{Corollary}
\newtheorem{lemma}[theorem]{Lemma}
\newtheorem{prop}[theorem]{Proposition}
\theoremstyle{remark}

\numberwithin{equation}{section}
\newcommand{\tr}{\operatorname{tr}}
\newcommand{\otr}{\operatorname{0-tr}}
\newcommand{\vol}{\operatorname{Vol}}

\newcommand{\re}{\operatorname{Re}}
\newcommand{\im}{\operatorname{Im}}
\newcommand{\res}{\operatorname{Res}}

\newcommand{\supp}{\operatorname{supp}}

\newcommand{\bbR}{\mathbb{R}}
\newcommand{\bbH}{\mathbb{H}}

\newcommand{\bbC}{\mathbb{C}}
\newcommand{\bbZ}{\mathbb{Z}}
\newcommand{\bbN}{\mathbb{N}}
\newcommand{\calR}{\mathcal{R}}

\newcommand{\calZ}{\mathcal{Z}}
\newcommand{\cinf}{C^\infty}
\newcommand{\del}{\partial}

\newcommand{\vep}{\varepsilon}

\newcommand{\nh}{\tfrac{n}{2}}
\newcommand{\norm}[1]{\left\Vert #1 \right\Vert}
\newcommand{\abs}[1]{\left|#1 \right|}
\newcommand{\brak}[1]{\langle #1 \rangle}
\newcommand{\Ai}{{\rm Ai}}

\begin{document}

\title[Resonance asymptotics]{Resonance asymptotics for asymptotically
hyperbolic manifolds with warped-product ends}
\author[Borthwick]{David Borthwick}
\thanks{Borthwick supported in part by NSF\ grant DMS-0901937.}
\author[Philipp]{Pascal Philipp}
\date{\today}

\begin{abstract}
We study the spectral theory of asymptotically hyperbolic manifolds with ends of warped product type.  Our main result is an upper bound on the resonance counting function with a geometric constant expressed in terms of the respective Weyl constants for the core of the manifold and the base manifold defining the ends.
\end{abstract}
\maketitle

\section{Introduction}

In this paper we study the spectral theory of an
\emph{asymptotically hyperbolic manifold with warped-product ends} $(X,g)$, with $\dim X = n+1$, $n\ge 1$. 
By this we mean that $X$ admits a decomposition
\[
X = K \sqcup X_0,
\]
where $X_0 = (0,1] \times \Sigma$ with $(\Sigma,h)$ a compact Riemannian
manifold without boundary,
\begin{equation}\label{wp.metric}
g|_{X_0} = \frac{dx^2 + h}{x^2},
\end{equation}
and $K$ is a compact manifold with boundary $\del K \simeq \Sigma$.  We allow
$\Sigma$ to be disconnected, so that multiple ends can be considered without
changing the notation.

For a general conformally compact, asymptotically hyperbolic manifold, Joshi-S\'a Barreto 
\cite{JS:2000} proved the existence of a product decomposition near infinity with
a metric of the form \eqref{wp.metric} with $h = h(x,y,dy)$, meaning that $h$ could depend on $x$.  
Our restriction to warped-product ends amounts to taking a fixed metric $h$ independent of $x$.

In the $n=1$ case, $\Sigma$ could only be a circle, and the model $X_0$ is isometric to the
flared end of the parabolic cylinder $\bbH^2/\brak{z \mapsto z+1}$.  In higher dimensions $X_0$ will generally 
not have constant curvature.

Since \eqref{wp.metric} implies in particular that $g$ is an \emph{even} asymptotically hyperbolic metric, the resolvent $R_g(s) := (\Delta - s(n-s))^{-1}$ admits a meromorphic continuation to $s\in \bbC$, with poles of finite rank, by Mazzeo-Melrose \cite{MM:1987} and Guillarmou \cite{Gui:2005a}.  We define the resonance set $\calR_g$ to be the set of poles of $R_g(s)$, repeated according to multiplicity.  The corresponding resonance counting function is
\begin{equation}\label{Ndef}
N_g(t) := \#\left\{\zeta\in\calR_g:\>\abs{\zeta-\nh} \le t \right\}.
\end{equation}

In the full asymptotically hyperbolic setting, we know essentially nothing of the resonance set beyond the meromorphic continuation result that allows its definition.  At this level of generality, we have no bounds on $N_g(t)$ and no existence results for $\calR_g$.  The only general information we have on resonance distribution is a result of Guillarmou \cite{Gui:2005} gives exponentially thin resonance-free regions near the critical line $\re s = \nh$.   All of the current resonance counting results for asymptotically hyperbolic metrics actually assume that the sectional curvature is constant outside a compact set.  (This allows a more direct construction of the parametrix for the resolvent than in the general case.)  Under this stronger assumption, we have $N_g(t) \asymp t^{n+1}$, as well as a Poisson-type trace formula expressing the regularized wave trace as a sum over the resonance set \cite{GZ:1995b, CV:2003, Borthwick:2008}.

In this paper, we establish the following (relative) Poisson formula for resonances:
\begin{theorem}\label{rel.poisson}
Assume $(X,g)$ is an asymptotically hyperbolic manifold with warped-product ends.  Let $\Delta_0$ denote the Laplacian
with Dirichlet boundary conditions on the model end $(X_0,g)$, and $\calR_0$ the corresponding resonance set.  The difference of the regularized wave traces satisfies
\[
\begin{split}
&\otr \left[ \cos \left(t \sqrt{\Delta_{g} - \tfrac{n^2}4}\,\right) \right] 
- \otr \left[ \cos \left(t \sqrt{\Delta_{0} - \tfrac{n^2}4}\,\right) \right]  \\
&\qquad = \frac12 \sum_{\zeta\in \calR_g} e^{(\zeta-\frac{n}2)|t|} 
- \frac12 \sum_{\zeta\in \calR_0} e^{(\zeta-\frac{n}2)|t|},
\end{split}
\]
in the sense of distributions on $\bbR - \{0\}$.
\end{theorem}

Here, as in \cite{GZ:1997, Borthwick:2008}, the $0$-trace is a formal trace defined as the Hadamard finite part for $\vep\to 0$ of the integral over $\{x \ge \vep\}$ of the restriction of the kernel to the diagonal.

For asymptotically hyperbolic $(X,g)$ we can define a scattering matrix $S_g(s)$ as in \cite{JS:2000}.  This is a family of pseudodifferential operators on $\Sigma$, meromorphic in $s\in \bbC$. For such metrics the relationship between resonances and poles of $S_g(s)$ was established in \cite{GZ:1997, BP:2002, Gui:2005}.  In particular, Guillarmou~\cite{Gui:2005} showed that $S_g(s)$ may have `conformal' poles  
$s\in \nh -\bbN$ which do not correspond to resonances.  However, in the case of asymptotically hyperbolic metrics of warped-product type, we will see that these conformal poles are ruled out in any dimension.  Hence the multiplicities of scattering poles agree with those of the resonance set, except possibly at the finitely many points $s$ where $s(n-s)$ is a discrete eigenvalue of $\Delta_g$.

One application of the Poisson formula of Theorem~\ref{rel.poisson} is a Weyl asymptotic for the \emph{relative scattering phase}, which is defined as the log of the Fredholm determinant of $S_g(s) S_0(s)^{-1}$ (see Corollary~\ref{scphase.cor}).  Because of the connection between resonances and scattering poles, we can use this asymptotic in conjunction with a contour integral involving $\det S_g(s) S_0(s)^{-1}$ to produce a precise upper bound on the resonance counting function.

To state the result, we introduce the classical Weyl constants for the compact manifolds $K$ and $\Sigma$:
\[
W_K := \frac{\vol(K, g)}{(4\pi)^{\frac{n+1}2}\Gamma(\frac{n+3}2)} ,  \quad
W_\Sigma :=  \frac{\vol(\Sigma, h)}{(4\pi)^{\frac{n}2}\Gamma(\frac{n+2}2)} .
\]
For $\arg \alpha \in [0,\tfrac{\pi}2]$, define
\[
\rho(\alpha) := \sqrt{\alpha^2+1} + \alpha \log \left(\frac{i}{\alpha + \sqrt{\alpha^2+1}}\right).
\]
Denote by $\alpha_0\approx 1.509$ the point where $\{\re \rho(\alpha) = 0\}$ meets the real axis, and let the curve $\gamma$ be defined as  the portion of $\{\re \rho(\alpha) = 0\}$ that connects $i$ and $\alpha_0$ (c.f.~Figure \ref{gamma.fig}). 

\begin{theorem}\label{main.thm}
For $(X,g)$ an asymptotically hyperbolic metric of warped-product type, 
\begin{equation}\label{tN.bound}
(n+1)\int_0^a \frac{N_g(t)}{t} \: dt \le \Bigl[2W_K + c_n W_\Sigma \Bigr] a^{n+1} + o(a^{n+1}),
\end{equation}
where the dimensional constant is
\begin{equation}\label{constants}
\begin{split}
c_n & := \frac{2n}{(n+1)\pi} \int_{\gamma}
\frac{\abs{\rho'(\alpha)}}{\abs{\alpha}^{n+1}} \: \abs{d\alpha} 
+ \frac{\alpha_0^{-n}}{n+1}  \\
&\qquad + \frac{n(n+1)}{\pi} \int_{-\tfrac\pi2}^{\tfrac\pi2} \int_0^\infty \frac{[-\re \rho(xe^{i\abs\theta})]_+}{x^{n+2}} \: dx \: d\theta.
\end{split}
\end{equation}
\end{theorem}

The integrated counting function that appears in \eqref{tN.bound} is common usage in applications of Jensen's formula in complex analysis.
The bound \eqref{tN.bound} implies a corresponding bound for $N_g(t)$, at the cost of an extra factor of $e$ in the constant. 
(An asymptotic for the integrated form would be equivalent to an asymptotic for $N_g(t)$ with the same constant.)

To prove these results we will first establish a suboptimal bound on the growth of $N_g(t)$ using the Fredholm determinant method.
Using this crude estimate of the order of growth, we can apply the methods used in the hyperbolic-near-infinity case in \cite{Borthwick:2008} to prove Theorem~\ref{rel.poisson}.  To prove Theorem~\ref{main.thm}, we first develop an exact asymptotic for the model counting function $N_0(t)$, which is (from Proposition~\ref{modelcount.prop})
\[
N_0(t) = \left[ \frac{2n}{(n+1)\pi} \int_{\gamma} \frac{\abs{\rho'(\alpha)}}{\abs{\alpha}^{n+1}} \: d\abs\alpha 
+ \frac{\alpha_0^{-n}}{n+1} \right] W_\Sigma \>t^{n+1} + O\big(t^{n+\frac13}\big).
\]
Then we use the relative counting formula provided by a contour integral of the scattering determinant to obtain the sharper estimate for $N_g(t)$.

\section{The model case}\label{model.sec}

The model space is $X_0 := (0,1]\times \Sigma$, where $(\Sigma, h)$ is a
compact $n$-dimensional Riemannian manifold without boundary.  The metric on $X_0$ is the warped-product
\[
g_0 := \frac{dx^2 + h}{x^2}.
\]
The corresponding Laplacian is
\[
\Delta_0 = - (x\del_x)^2 + nx\del_x + x^2 \Delta_h.
\]
We can take the boundary to be $\{x = 1\}$ without loss of generality.  By the scale-invariance of the $dx^2$ component of $g_0$, imposing the boundary condition at some other value $x = b$ would be equivalent to rescaling $h \leadsto b^{2}h$.

\subsection{Spectral operators}\label{sop.sec}
Suppose that $\{\phi_\lambda\}$ is a complete set of eigenfunctions for
$\Delta_h$, with the convention
\[
\Delta_h \phi_\lambda = \lambda^2 \phi_\lambda.
\]
For $w = u(x) \phi_\lambda$, the equation $(\Delta_0 - s(n-s)) w = 0$ translates
to the coefficient equation
\begin{equation}\label{coeff.eq}
\Bigl[ - (x\del_x)^2 + nx\del_x + \lambda^2 x^2  - s(n-s)\Bigr]u = 0.
\end{equation}
This is a modified Bessel equation, with the Bessel parameter given by
\[
\nu := s - \nh.
\]
To simplify formulas, we will make this identification throughout this section and switch freely between $s$ and $\nu$.

The general solution to \eqref{coeff.eq} is a linear combination of the terms $x^{\frac{n}2} I_{\pm\nu} (\lambda x)$ for $\lambda>0$ and $x^{\frac{n}2\pm \nu}$ for $\lambda=0$.
As $x\to 0$ the Bessel function has asymptotic
\begin{equation}\label{Iasymp.0}
I_\nu(\lambda x) \sim  \frac{1}{\Gamma(\nu+1)} \left( \frac{\lambda
x}{2}\right)^{\nu},
\end{equation}
for $\nu \notin - \bbN$.
For future use we single out the `outgoing' solutions
\[
\begin{split}
u_\lambda^+(s;x) &:= x^{\frac{n}2} I_\nu (\lambda x) \quad\text{for }\lambda \ne 0,\\
u^+_0(s;x) &:= x^s,
\end{split}
\] 
which have asymptotics proportional to $x^s$ as $x\to 0$.  
We will also need solutions satisfying the boundary condition at $x=1$,
\[
\begin{split}
u_\lambda^0(s;x) & := \frac{\Gamma(\nu) \Gamma(1-\nu)}{2} x^{\frac{n}2} \bigl[ I_{\nu} (\lambda) I_{-\nu} (\lambda
x) - I_{-\nu} (\lambda ) I_{\nu} (\lambda x) \bigr] \quad\text{for }\lambda \ne 0,\\
u^0_0(s; x) &:= \frac{1}{2\nu} \left[x^{n-s} - x^s\right].
\end{split}
\]
The Gamma factors are included in $u_\lambda^0(s)$ to cancel zeros that would otherwise occur at $\nu \in \bbZ$ (c.f.~\eqref{I.reflect}).  Similarly, $u^0_0(s)$ is not actually singular at $s=\nh$; it simply takes on the limiting value $u^0_0(\nh; x) := - x^{n/2} \log x$.

We can easily express the model resolvent, Poisson kernel, and scattering matrix
in terms of the solutions $u_\lambda^+$ and $u_\lambda^0$.  

\subsubsection{Resolvent}
With respect to the eigenbasis $\{\phi_\lambda\}$ for $\Sigma$, the kernel of
the resolvent can be written
\begin{equation}\label{R0.decomp}
R_0(s; x, \omega, x', \omega') := \sum_\lambda a_\lambda(s;x,x')
\phi_\lambda(\omega) \overline{\phi_\lambda(\omega')},
\end{equation}
where the coefficients satisfy
\[
\Bigl[ - (x\del_x)^2 + nx\del_x + \lambda^2 x^2 - s(n-s)\Bigr] a_\lambda(s; x,
x') =  x^{n+1}\delta(x-x'),
\]
with boundary conditions $a_\lambda(s;x,x') \sim c(s,x') x^s$ at $x=0$ and $a_\lambda(s;1,x') = 0$.  The unique
solution satisfying these conditions is
\[
a_\lambda(s;x,x') = A_\lambda(s) \begin{cases}u_\lambda^+(s; x) u_\lambda^0(s; x') & x
\le x' \cr
u_\lambda^0(s; x) u_\lambda^+(s; x') & x \ge x', \end{cases}
\]
with 
\[
A_\lambda(s) := \frac{1}{I_\nu(\lambda)},
\]
for $\lambda \not= 0$ and 
\[
A_0(s) :=  1.
\]
From the explicit formula for $a_\lambda(s;x,x')$ we can read off the model resonance set,
\begin{equation}\label{R0.def}
\calR_0 = \bigcup_{\substack{\lambda^2 \in \sigma(\Delta_h)\\ \lambda\not=0}} 
\left\{s \in \bbC:\> I_{s-\frac{n}2}(\lambda) = 0 \right\}.
\end{equation}
Since $I_\nu(z)$ is nonzero for $z >0$ and $\re \nu \ge 0$, the resonance set lies completely in the half-plane 
$\re s < \nh$.  An example of the model resonance set is shown in Figure~\ref{X0Plot.fig}.
\begin{figure}
\begin{center}
\begin{overpic}[scale=.7]{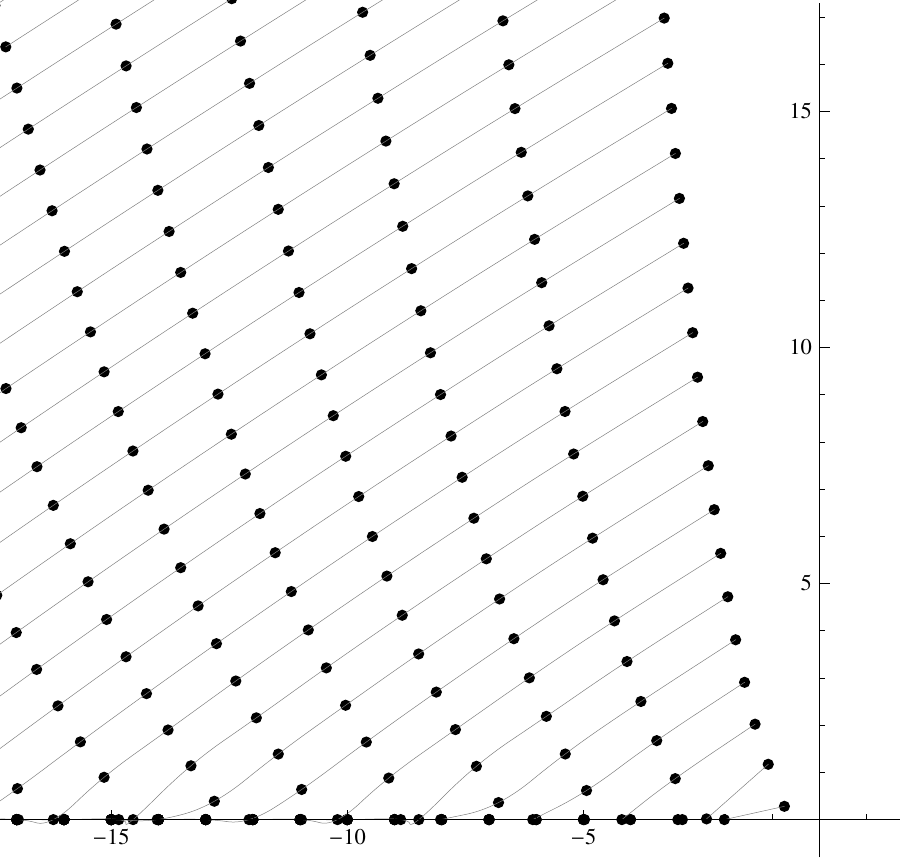}
\end{overpic}
\end{center}
\caption{Resonance plot for the model case $X_0 = (0,1] \times S^2$.  The thin lines indicate the spherical harmonic mode $l$,
starting from $l = 1$ in the bottom right corner.  The multiplicity on each line is $2l+1$.}\label{X0Plot.fig}
\end{figure}

\subsubsection{Poisson operator}\label{poisson.op.sec}

The Poisson operator $E_0(s)$ maps functions on $\Sigma$ to generalized eigenfunctions of $\Delta_0$ on $X_0$.
Its kernel is obtained from the resolvent by
\[
E_0(s; z,\omega') := \lim_{x'\to 0} {x'}^{-s} R_0(s;z,z'),
\]
where $z' = (x',\omega')$.  We can thus derive from \eqref{R0.decomp} the decomposition
\[
E_0(s; x, \omega, \omega') := \sum_\lambda b_\lambda(s;x) \phi_\lambda(\omega)
\overline{\phi_\lambda(\omega')},
\]
where by \eqref{Iasymp.0} we have
\begin{equation}\label{bl.def}
\begin{split}
b_\lambda(s;x) & = \frac{1}{\Gamma(\nu+1)} \left(\frac{\lambda}{2}\right)^{\nu} \frac{u_\lambda^0(x)}{I_\nu(\lambda)}\quad\text{for }\lambda \ne 0, \\
b_0(s;x) &= \frac{1}{2\nu} \left[ x^{n-s} - x^s \right].
\end{split}
\end{equation}

\subsubsection{Scattering matrix}

The scattering matrix can be derived from the Poisson operator through the two-part asymptotic, for $f \in \cinf(\Sigma)$ and $s\notin \bbZ/2$,
\[
(2s-n) E(s) f \sim x^{n-s}f + x^s S_0(s) f,
\]
as $x \to 0$.  In the model case $S_0(s)$ is diagonalized by the eigenfunctions $\{\phi_\lambda\}|_{\lambda \in \sigma(\Delta_h)}$, and we will use $[S_0(s)]_\lambda$ to denote the corresponding eigenvalue.

From \eqref{bl.def}, using \eqref{Iasymp.0}, we derive
\begin{equation}\label{S0.def}
\begin{split}
[S_0(s)]_\lambda &= \left( \frac{\lambda}{2}\right)^{2\nu}
\frac{\Gamma(-\nu)}{\Gamma(\nu)} 
\frac{I_{-\nu}(\lambda)}{I_\nu(\lambda)}\quad\text{for }\lambda \ne 0, \\
[S_0(s)]_0 &= - 1.
\end{split}
\end{equation}

The scattering poles $(X, g_0)$ are defined as the poles of the normalized scattering matrix
\[
\tilde S_0(s) := \frac{\Gamma(s-\frac{n}2)}{\Gamma(\frac{n}2-s)} S_0(s).
\]
In general the set of scattering poles could differ from the resonance set at certain points, but for the model case we see that the set of scattering poles is also given by $\calR_0$.

\subsection{Bessel function estimates}\label{bessel.sec}

For bounded $z > 0$, we can estimate $I_\nu(z)$ easily from the series definition,
\begin{equation}\label{Iseries}
I_\nu(z) :=  \sum_{k=0}^{\infty} \frac{(z/2)^{\nu+2k}}{k! \Gamma(\nu+k+1)},
\end{equation}
which converges for all $\nu\in\bbC$.  Although it is a power series in $z$, this could also be viewed as an expansion in $\nu$ for $z$ fixed.  If $z$ is restricted to a compact interval away from zero, we have a uniform bound
\begin{equation}\label{Iseries.bnd}
I_\nu(z) = \frac{(z/2)^{\nu}}{\Gamma(\nu+1)}(1+ O(\nu^{-1})).
\end{equation}
This will cover the estimation of $I_\nu(\lambda x)$ where $\lambda$ lies in some bounded interval. 

For cases where $\lambda$ is large, we use the method from Olver \cite[\S11.10]{Olver}.   
Set
\begin{equation}\label{rho.def}
\rho(\alpha, x) := \sqrt{\alpha^2+x^2} + \alpha \log \left(\frac{i x}{\alpha +
\sqrt{\alpha^2+x^2}}\right),
\end{equation}
and
\[
\zeta := (\tfrac32 \rho)^\frac23.
\]
Frequently we will fix $x=1$ and then we simply write $\rho(\alpha) := \rho(\alpha, 1)$.
With $\alpha$ in the first quadrant, $\rho$ and $\zeta$ occupy the sectors $\arg \rho \in
(0,\frac{3\pi}2)$ and $\arg \zeta\in (0,\pi)$.  Note that $\zeta=0$ precisely when $\alpha = i x$. 
This corresponds to the \emph{turning point}
of the (transformed) Legendre equation.

\begin{prop}\label{IK.Ai.prop}  
For $\lambda$ sufficiently large we have
\[
\begin{split}
I_{\lambda \alpha}(\lambda x) &= \frac{2\pi^\frac12}{\Gamma(\lambda \alpha + 1)}
(-i\lambda\alpha)^{\lambda\alpha}
\alpha^{\frac12} e^{-\frac{\pi i}6} \lambda^{\frac16} e^{-\lambda \alpha} 
\left(\frac{\zeta}{\alpha^2+x^2}\right)^\frac14  \Ai\left(e^{-\frac{2\pi i}3}
\lambda^\frac23 \zeta \right) [1+O(\lambda^{-1})], \\
K_{\lambda \alpha}(\lambda x) &=  2^\frac12 \pi \lambda^{-\frac13} i^{\lambda\alpha} 
\left(\frac{\zeta}{\alpha^2+x^2}\right)^\frac14  \Ai\left( \lambda^\frac23 \zeta
\right) [1+O(\lambda^{-1})],
\end{split}
\]
uniformly for $\arg\alpha \in [0,\tfrac{\pi}2]$ and $x$ contained in a compact
interval of $\bbR_+$.
\end{prop}
\begin{proof}
From Olver \cite[\S11.10]{Olver} we obtain
\[
I_{\lambda \alpha}(\lambda x) = c_1(\lambda,\alpha)
\left(\frac{\zeta}{\alpha^2+x^2}\right)^\frac14 
\Ai\left(e^{-\frac{2\pi i}3} \lambda^\frac23 \zeta \right) [1+O(\lambda^{-1})].
\]
Note that as $x\to 0$, 
\[
\rho(\alpha, x) = \alpha\log x+ \alpha + \alpha \log \frac {i}{2\alpha} +
O(x^2).
\]
With $\alpha$ in the first quadrant, this limit takes $e^{-\frac{2\pi i}3} \zeta
\to \infty$ in the sector $(0,\tfrac{\pi}3)$. 
In this limit,
\[
 \left(\frac{\zeta}{\alpha^2+x^2}\right)^\frac14 
\Ai\left(e^{-\frac{2\pi i}3} \lambda^\frac23 \zeta \right) \sim
\frac{\alpha^{-\frac12}}{2\pi^\frac12} e^{\frac{\pi i}6} \lambda^{-\frac16} 
e^{\lambda [\alpha\log x+ \alpha + \alpha \log \frac {i}{2\alpha}]}.
\]
Comparing this to the asymptotic, as $x\to 0$,
\[
I_{\lambda \alpha}(\lambda x) \sim \frac{1}{\Gamma(\lambda \alpha + 1)}
\left(\frac{\lambda x}{2}\right)^{\lambda\alpha},
\]
we find that
\[
c_1(\lambda,\alpha) = \frac{2\pi^\frac12}{\Gamma(\lambda \alpha + 1)}
\left(\frac{\lambda}{2}\right)^{\lambda\alpha}
\alpha^{\frac12} e^{-\frac{\pi i}6} \lambda^{\frac16}  e^{-\lambda [\alpha +
\alpha \log \frac {i}{2\alpha}]}.
\]

Similarly, for $K_\nu$ we start from
\[
K_{\lambda \alpha}(\lambda x) = c_2(\lambda,\alpha)
\left(\frac{\zeta}{\alpha^2+x^2}\right)^\frac14 
\Ai\left(\lambda^\frac23 \zeta \right) [1+O(\lambda^{-1})].
\]
As $x\to \infty$, 
\[
\rho(\alpha,x) = x + \alpha \log i + o(1),
\]
so
\[
\left(\frac{\zeta}{\alpha^2+x^2}\right)^\frac14 \Ai\left(\lambda^\frac23 \zeta
\right) \sim
\frac{x^{-\frac12}}{2\pi^\frac12} \lambda^{-\frac16} e^{-\lambda(x + \alpha\log
i)},
\]
as $x\to\infty$.  On the other hand,
\[
K_{\lambda \alpha}(\lambda x) \sim \left( \frac{\pi}{2\lambda
x}\right)^{\frac12} e^{-\lambda x},
\]
as $x\to \infty$.  Thus
\[
c_2(\lambda, \alpha) = 2^\frac12 \pi \lambda^{-\frac13} e^{\lambda \alpha \log
i}.
\]
\end{proof}

The Airy function has zeros only on the negative real axis, with the first at $w \approx -2.338$.  For future reference we can thus note that
\begin{equation}\label{airy.zero}
\Ai(w) \asymp 1,\quad\text{for }\abs{w} \le 2.33.
\end{equation}
For large arguments we can use the well-known Airy function asymptotics \cite[\S4.4]{Olver}.  For $\abs{\arg w} < \pi-\delta$, we have the uniform estimate
\begin{equation}\label{ai.asym1}
\Ai(w) = \frac{1}{2\pi^{\frac12}} w^{-\frac14} e^{-\xi} \bigl[1 + O(|\xi|^{-1})\bigr],
\end{equation}
where $\xi := \tfrac23w^{\frac32}$, with the constant in the error term bounded by $c (\sin \delta/2)^{-1}$.  Most often we will use this simply to estimate
\begin{equation}\label{airy.sector}
\Ai(w) \asymp \brak{w}^{-\frac14} e^{-\xi}, \quad\text{for }\abs{\arg w} < \pi-\delta.
\end{equation}
We can extend this to the negative real axis using the identity
\begin{equation}\label{Ai.tri}
\Ai(w) = e^{\frac{\pi i}3} \Ai(e^{-\frac{2\pi i}3}w) +  e^{-\frac{\pi i}3} \Ai(e^{-\frac{4\pi i}3}w).
\end{equation}
From \eqref{ai.asym1} this yields
\begin{equation}\label{ai.asym2}
\Ai(w) = \frac{1}{2\pi^{\frac12}} w^{-\frac14} \left(\exp(-\tfrac23 w^{\frac32}) + i\exp(\tfrac23 w^{\frac32}) \right)
\bigl[1 + O(|w|^{-\frac32})\bigr],
\end{equation}
uniformly for $\arg w \in [\tfrac{\pi}3 + \delta, \tfrac{5\pi}3 - \delta]$.   

For the resulting estimates let us rescale $\rho$ to
\begin{equation}\label{psi.def}
\begin{split}
\psi(\nu,\lambda x) &:= \lambda \rho(\alpha,x) \\
&= \sqrt{\nu^2+\lambda^2 x^2} + \nu \log \left(\frac{i \lambda x}{\nu +
\sqrt{\nu^2+\lambda^2 x^2}}\right).
\end{split}
\end{equation}
Combining Proposition~\ref{IK.Ai.prop} with \eqref{ai.asym1}, \eqref{ai.asym2}, and Stirling's
formula yields the following:
\begin{corollary}\label{IK.cor}
For $\arg\nu\in [0,\tfrac{\pi}2]$, with $\lambda$, $\nu$ and $\psi = \psi(\nu,\lambda x)$ sufficiently large, we have
\begin{equation}\label{I.ntest}
I_{\nu}(\lambda x) = \frac{1}{\sqrt{2\pi}} (\nu^2+\lambda^2 x^2)^{-\frac14}  i^{-\nu}
e^{\psi} \left[1 + O(\psi^{-1}) + O(\lambda^{-1}) + O(\nu^{-1}) \right].
\end{equation}
Similarly, for $\arg (\nu - i\lambda x) \le \tfrac{\pi}2 - \vep$, corresponding to $\arg \zeta \in [0,\pi-\delta]$,
\begin{equation}\label{K.ntest}
K_{\nu}(\lambda x) =  \sqrt{\frac{\pi}2}
(\nu^2+\lambda^2 x^2)^{-\frac14} i^{\nu} e^{-\psi} 
\left[1 + O(\psi^{-1}) + O(\lambda^{-1}) \right].
\end{equation}
If $\arg (\nu - i\lambda x) \in [\tfrac{\pi}2 - \vep, \tfrac{\pi}2]$, then \eqref{K.ntest} holds with the
replacement
\[
e^{-\psi} \leadsto e^{-\psi} + ie^{\psi}.
\]  
(Under this condition, $\re \psi\le 0$, so the correction term is $O(1)$ and will not affect upper bounds for $K_\nu$.)
\end{corollary}

These estimates don't apply near the `turning point' of the transformed Bessel equation, where $\nu = i\lambda x$ and
$\psi = 0$.  
\begin{corollary}\label{turn.cor}
For $\arg\nu\in [0,\tfrac{\pi}2]$, suppose $\nu$ is close to $i\lambda x$ in the sense that $\abs{\psi} < c$ with $c$ sufficiently small.  For $x>0$ fixed, $\lambda$ sufficiently large, we have
\begin{equation}\label{IK.turn}
\begin{split}
I_{\nu}(\lambda x) &\asymp (\lambda x)^{-\frac13} i^{-\nu},
\\
K_{\nu}(\lambda x) &\asymp (\lambda x)^{-\frac13} i^{\nu}.
\end{split}
\end{equation}
\end{corollary}
\begin{proof}
To estimate near the turning point, suppose that $\alpha = ix + \eta$, and $\nu = \lambda \alpha$ as above.  
For $\eta$ sufficiently small and $x>0$ we have
\[
\rho \asymp x^{-\frac12} \eta^{\frac32}.
\]
This means $\psi \asymp \lambda x^{-\frac12} \eta^{\frac32}$, so that $\abs{\psi} \le c$ corresponds to $\abs{\eta} \le c \lambda^{-\frac23} x^\frac13$.  

Consider the estimates of Proposition~\ref{IK.Ai.prop}.  Since $\zeta \asymp x^{-\frac13} \eta$, the assumption 
$\abs{\psi} \le c$ means that the argument of the Airy functions, $\lambda^\frac23 \zeta$, is bounded near $0$.
Note that $\Ai(0) \ne 0$, and that 
\[
\frac{\zeta}{\alpha^2+x^2} \asymp x^{-\frac43},
\]
for $\eta$ sufficiently small.  The estimate on $K_{\nu}(\lambda x)$ then follows immediately from  Proposition~\ref{IK.Ai.prop}.
For $I_{\nu}(\lambda x)$ we must also apply Stirling's formula.  This is justified for large $\lambda$ since $\abs{\psi}\le c$ implies $\abs{\nu} \asymp \lambda$.
\end{proof}

\bigbreak
In addition to the estimates given above for $I_\nu$, $K_\nu$ with $\re \nu \ge 0$, we will need to be able to control the ratio 
$I_{-\nu}/I_{\nu}$, which appears, for example, in the scattering matrix.  We can derive these from the results above using the identity 
\begin{equation}\label{I.reflect}
I_{-\nu}(z) = I_{\nu}(z) + \frac{2 \sin \pi\nu}{\pi} K_\nu(z).
\end{equation}
To analyze the ratio $I_{-\nu}/I_{\nu}$, we note that  using Proposition~\ref{IK.Ai.prop}, with Stirling's formula applied to $\Gamma(\nu+1)$ for large $\nu$, implies
\begin{equation}\label{KI.ratio}
\frac{2 \sin \pi\nu}{\pi} \frac{K_{\nu}(\lambda x)}{I_{\nu}(\lambda x)}  = e^{\frac{2\pi i}3} (1 - e^{2\pi i \nu}) 
\frac{\Ai((\frac32 \psi)^{2/3})}{\Ai((-\frac32 \psi)^{2/3})}
\bigl[1+ O(\lambda^{-1}) + O(\nu^{-1})\bigr],
\end{equation}
for $\lambda$ sufficiently large.
We first consider the estimates away from the zeros of $I_{-\nu}(\lambda x)$.

\begin{figure}
\begin{tabular}{cc}
\begin{overpic}[scale=.60]{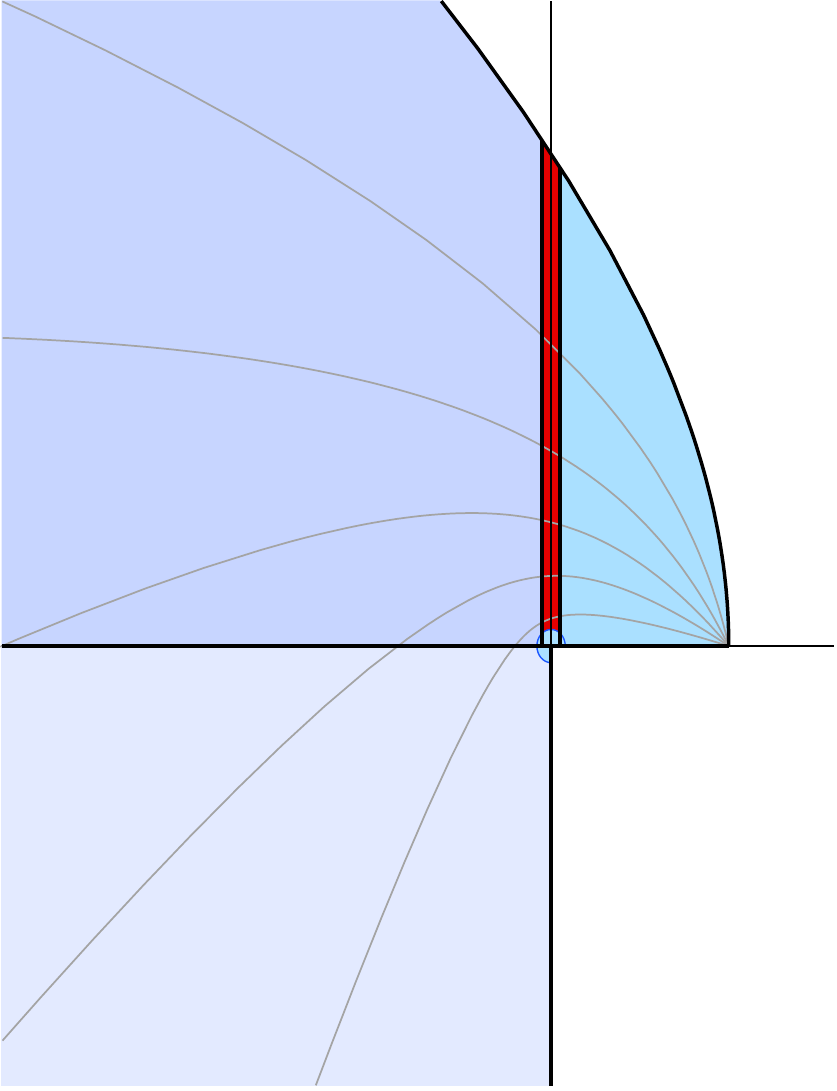}
\put(56,90){$\psi$}
\end{overpic} &\phantom{MMM}
\begin{overpic}[scale=.75]{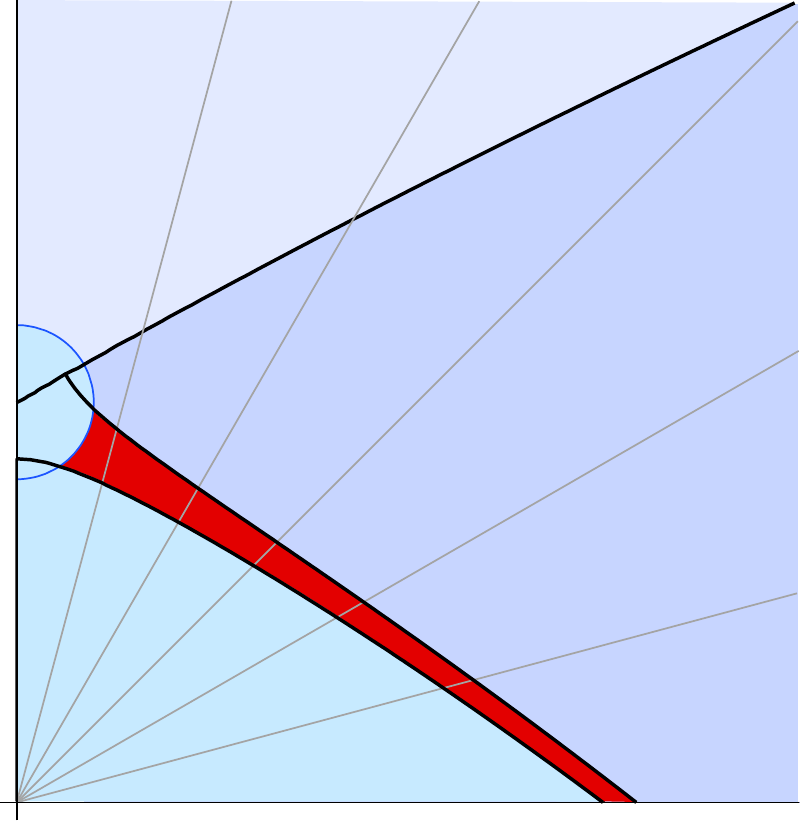}
\put(-7,50){$i\lambda x$}
\put(100,90){$\nu$}
\put(15,80){$\im \psi \le 0$}
\put(50,55){$\re \psi \le -b$}
\put(50,48){$\im \psi \ge 0$}
\put(15,15){$\re \psi \ge b$}
\end{overpic} 
\end{tabular}
\caption{Regions for the estimates in Lemmas~\ref{Iratio.lemma} and \ref{Iratio.z}.  The red zone contains the non-trivial zeros of $I_{-\nu}(\lambda x)$.}\label{NuPsiPlot.fig}
\end{figure}
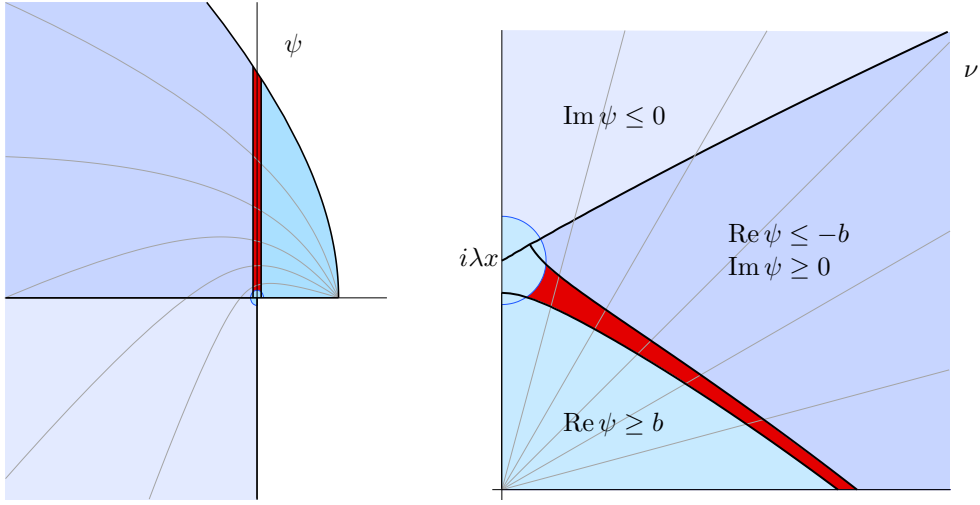

\begin{lemma}\label{Iratio.lemma}
For the estimates below we assume that $\arg\nu\in [0,\tfrac{\pi}2]$ and $\lambda>M$, with $M$ large enough that the estimates from Proposition~\ref{IK.Ai.prop} apply, and that $x>0$.  There exist constants $\delta>0$ and $c>b>0$ such that:
\begin{enumerate}
\item  For either $\re \psi \ge b$ or $\abs{\psi} < c$,
\begin{equation}\label{II.asymp1}
\frac{I_{-\nu}(\lambda x)}{I_\nu(\lambda x)} \asymp 1,
\end{equation}
with constants that depend only on $M$, $b$, and $c$.
\item For $\im \psi \ge 0$, $\re \psi \le - b$ and (for the lower bound) $d(\nu, \bbN_0) \ge \delta$,
\begin{equation}\label{II.asymp2}
\frac{I_{-\nu}(\lambda x)}{I_\nu(\lambda x)} \asymp e^{-2\psi},
\end{equation}
with constants that depend only on $M$, $b$ and $\delta$.
\item For $\im \psi \le 0$ (which occurs only when $\re \psi \le 0$ also),
\begin{equation}\label{II.asymp3}
\frac{I_{-\nu}(\lambda x)}{I_\nu(\lambda x)} \asymp e^{-2\psi},
\end{equation}
with constants that depend only on $M$ and $x$.
\end{enumerate}
\end{lemma}
\begin{proof}
(1)  We can apply the Airy estimate \eqref{airy.sector} to \eqref{KI.ratio} to obtain the bound, 
\begin{equation}
\frac{2 \sin \pi\nu}{\pi} \frac{K_{\nu}(\lambda x)}{I_{\nu}(\lambda x)} \preceq e^{-2\psi},
\end{equation}
for $\lambda \ge M$, $\im \nu \ge 0$, and $\abs{\nu} \ge \tfrac12$, with constants that depend on $M$ and $b$.  In this case, for $\re \psi \ge b$ with $b \gtrsim 0.4$, \eqref{II.asymp1} follows from \eqref{I.reflect}. 
For $\abs{\nu} < \tfrac12$, we can use the recurrence relation $I_{\nu}(z) = I_{\nu+2}(z) + 2z^{-1} (\nu+1) I_{\nu+1}(z)$ to move the estimates away from $\nu=0$ and obtain the same result.

Near the turning point, i.e.~with $\abs{\psi} < c$, we note that $\im \nu \asymp \lambda x$.  In this case, \eqref{II.asymp1} follows from \eqref{KI.ratio} and \eqref{airy.zero}, provided $c \lesssim 2.3$.

(2) 
For $d(\nu, \bbN_0) \ge \delta$ and $b > - (\log \sin \delta)/2$, the ratio \eqref{KI.ratio} dominates the $I_{-\nu}/I_\nu$ ratio and \eqref{II.asymp2} follows.

(3)  The assumption on $\psi$ corresponds to $\arg \zeta \in [\tfrac{2\pi}3, \pi]$, and also guarantees that $\im\nu \ge \lambda x$ (see Figure~\ref{NuPsiPlot.fig}).  By \eqref{KI.ratio} and \eqref{Ai.tri} this implies
\begin{equation}\label{KI.atri}
\frac{2\sin \pi\nu}{\pi} \frac{K_{\nu}(\lambda x)}{I_{\nu}(\lambda x)} = 
\left( -1 + e^{\frac{\pi i}3} \frac{\Ai(e^{\frac{2\pi i}3}(\frac32 \psi)^{2/3})}{\Ai((-\frac32 \psi)^{2/3})} \right)[1+O(\lambda^{-1})].
\end{equation}
Thus by \eqref{I.reflect} we have
\[
\frac{I_{-\nu}(\lambda x)}{I_\nu(\lambda x)} = e^{\frac{\pi i}3} \frac{\Ai(e^{\frac{2\pi i}3}(\frac32 \psi)^{2/3})}{\Ai((-\frac32 \psi)^{2/3})} [1+O(\lambda^{-1})] + O(\lambda^{-1}).
\]
The estimate \eqref{II.asymp3} now follows from \eqref{airy.sector}.
\end{proof}

\bigbreak
Lemma~\ref{Iratio.lemma} leaves out a region where $\abs\psi\ge c,\:\im \psi \ge 0$ and $\abs{\re \psi} \le b$, as illustrated in Figure~\ref{NuPsiPlot.fig}.  In this zone lower bounds are more delicate because it contains a non-trivial portion of the zero set
\[
\calZ_{\lambda x} := \left\{ \nu:\>  I_{-\nu}(\lambda x) = 0\right\}.
\]

\begin{lemma}\label{Iratio.z}
Assume that $\arg\nu\in [0,\tfrac{\pi}2]$ and $\lambda>M$, with $M$ large enough that the estimates from Proposition~\ref{IK.Ai.prop} apply, and that $\im \psi \ge 0$ and $\abs{\re \psi} \le b$.  Then
\[
\abs{\frac{I_{-\nu}(\lambda x)}{I_\nu(\lambda x)}} \le C_{M,b}.
\]
If in addition we assume that $d(\nu, \calZ_{\lambda x}) \ge \brak{\nu}^{-\beta}$ for some $\beta>0$, then
\[
\log \abs{\frac{I_{-\nu}(\lambda x)}{I_\nu(\lambda x)}} \ge -c_{M,b,\beta} \abs{\nu} \log \abs{\nu}.
\]
\end{lemma}
\begin{proof}
By the estimates in Lemma~\ref{Iratio.lemma}, we can see that for $\lambda \ge M$ with $M$ sufficiently large,
\[
\abs{\frac{I_{-\nu}(\lambda x)}{I_\nu(\lambda x)}} \preceq 1,
\]
for $\nu$ on the boundary of the region in question, with constants that depend only on $M$ and $b$.  The upper bound follows immediately.

For the lower bound we apply the minimum modulus theorem in the form \cite[Thm~1.11]{Levin} to $f(\nu) := I_\nu(\lambda x)/I_0(\lambda x)$ (normalized so $f(0) = 1$).  For $\eta>0$ sufficiently small and $m>0$ fixed,  inside the disk $\abs{\nu} \le m \lambda$, but excluding a set of disks whose radii sum to at most $4m\eta \lambda$, we have
\begin{equation}\label{levin.est}
\log \abs{\frac{I_\nu(\lambda x)}{I_0(\lambda x)}} > - \left(3 + \log \frac{3}{2\eta}\right) 
\log \left(\sup_{\abs{z} = 2me\lambda} \abs{\frac{I_z(\lambda x)}{I_0(\lambda x)}} \right).
\end{equation}
Since
\[
\abs{\re \psi(\nu,\lambda x)} = O(\lambda), \quad\text{for }\abs{\nu} \le C\lambda,
\]
we can apply Corollary \ref{IK.cor} (or Corollary \ref{turn.cor} in case $2me$ is close to $1$) and \eqref{I.reflect} to deduce that for any $m>0$, 
\[
\log \abs{I_\nu(\lambda x)} \le C_m \lambda,\quad\text{for }\abs{\nu} \le 2me \lambda,
\]
for $\lambda$ sufficiently large.  For the $I_0$ term the standard Bessel function asymptotic gives $I_0(\lambda x) \sim (2\pi \lambda x)^{-1/2} e^{\lambda x}$.  Combining these estimates with \eqref{levin.est} thus gives a lower bound
\begin{equation}\label{Inu.lower}
\log \abs{I_\nu(\lambda x)} > - c_m (1 + \log \eta^{-1}) \lambda,
\end{equation}
for $\abs{\nu} \le m \lambda$, excluding a set of disks whose radii sum to at most $4m\eta \lambda$.

Now we wish to apply the estimate to the region described in the lemma, in which $\abs{\nu} \asymp \lambda$ and $d(\nu, \calZ_{\lambda x}) \ge \brak{\nu}^{-\beta}$.  We can fix $m$ independently of $\lambda$ and choose $\eta = \kappa\lambda^{-\beta-1}$.  For $\kappa$ sufficiently small, the hypotheses of \eqref{Inu.lower} will be satisfied for all $\nu, \lambda$ in the region of interest.  For $\lambda$ sufficiently large, the claimed lower bound then follows from \eqref{Inu.lower}, with the extra $\log \abs{\nu}$ coming from the variable choice of $\eta$.
\end{proof}

\subsection{Spectral operator estimates}

We can now apply the estimates from \S\ref{bessel.sec} to the formulas for the model resolvent, Poisson operator, and scattering matrix from \S\ref{sop.sec}.  For the resolvent, we only need estimates in the physical half-plane, $\re s\ge \nh$.

\begin{prop}\label{R0est.prop}
Suppose $\chi_1,\chi_2 \in \cinf_0(0,1)$ are cutoff functions with disjoint
supports and $\sigma \ge 0$.  Then for $\re(s-\nh) \ge \vep$, we have
\[
\norm{\chi_1 R_0(s) \chi_2}_{\mathcal{L}(H^0,H^\sigma)} \le
C_{\vep,\sigma} \brak{s}^{-1+\sigma}.
\]
For $0 \le \re(s-\nh) \le \vep$, with $\abs{s - \nh} \ge \vep$, we have
\[
\norm{\chi_1 R_0(s) \chi_2}_{\mathcal{L}(H^0,H^\sigma)} \le
C_{\vep,\sigma}\brak{s}^{-\frac23+\sigma}.
\]
\end{prop}
\begin{proof}
By a standard argument involving resolvent identities, it suffices to prove the
estimates for $\sigma=0$
(see, e.g.~\cite[Lemma~9.8]{Borthwick}).

The first bound depends only on the location of the spectrum, 
$\sigma(\Delta_{0}) = [\tfrac{n^2}4, \infty)$.
From the spectral theorem and the fact that
\[
d(s(n-s), \sigma(\Delta_0)) = \begin{cases}\abs{s-\nh}^2 & \re (s-\nh) \ge
\abs{\im s} \\
2\abs{\re(s-\nh) \im s} & \re (s-\nh) \le \abs{\im s},  \end{cases}
\]
we find that
\[
\norm{R_0(s)} \le C_\vep \brak{s}^{-1},
\]
for $\re s \ge \nh + \vep$.

For the bound near the critical line we turn to the decomposition
\eqref{R0.decomp}.  Since the cutoffs yield a smoothing operator with compactly
supported coefficients, it suffices to obtain pointwise estimates of the
coefficients $a_\lambda$.  For $x_1< x_2$ we have
\begin{equation}\label{asig.I}
a_\lambda(s;x_1,x_2) = 2^{-1} \Gamma(\nu) \Gamma(1-\nu) (x_1 x_2)^{\frac{n}2}
I_\nu(\lambda x_1)
\left[ I_{-\nu} (\lambda x_2) - \frac{I_{-\nu} (\lambda)}{I_{\nu} (\lambda)} I_{\nu} (\lambda x_2)  \right],
\end{equation}
or, using \eqref{I.reflect},
\begin{equation}\label{asig.K}
a_\lambda(s;x_1,x_2) =  (x_1 x_2)^{\frac{n}2} I_\nu(\lambda x_1)
\left[ K_{\nu} (\lambda x_2) - \frac{K_{\nu} (\lambda)}{I_{\nu} (\lambda)} I_{\nu} (\lambda x_2)\right].
\end{equation}

The case where $\lambda$ is bounded is easily dealt with. For $\abs{\re \nu} \le \vep$ we can apply \eqref{Iseries.bnd}
directly in \eqref{asig.I} to obtain
\[
a_\lambda(s;x_1,x_2) = O(\brak{\im s}^{-1}),
\]
for $0\le \lambda \le M$.

For the rest of the proof we may assume that $\lambda \ge M$ such
that the estimates of Proposition~\ref{IK.Ai.prop} apply.
First we consider the case away from the turning
point.  That is, we assume 
$\abs{\nu-i \lambda x} \ge c \lambda^{\frac13}$ for all of $x = 1,\, x_1$ or $x_2$.  Then
 \eqref{I.ntest} and \eqref{K.ntest} directly in \eqref{asig.K}, giving the
estimate
\begin{equation}\label{asig.caseb}
\begin{split}
\abs{a_\lambda(s;x_1,x_2)} & \le C \abs{\nu^2+(\lambda x_1)^2}^{-\frac14}
\abs{\nu^2+(\lambda x_2)^2}^{-\frac14} \\
&\qquad \times\left(e^{\re[\psi(\nu,\lambda x_1) - \psi(\nu,\lambda x_2)]}
+ e^{\re[\psi(\nu,\lambda x_1) + \psi(\nu,\lambda x_2) - 2\psi(\nu,\lambda)]} \right).
\end{split}
\end{equation}
Since
\[
\del_x \psi(\nu,\lambda x) = \frac{\sqrt{\nu^2+\lambda^2x^2}}{x},
\]
we observe that $\re \psi$ is an increasing function of $x$ for $\re \nu\ge 0$. 
Thus the final expression in \eqref{asig.caseb} is $O(1)$.  Under these assumptions we conclude that 
\[
\abs{a_\lambda(s;x_1,x_2)} = O(\brak{\im s}^{-1}),
\]
uniformly in $\lambda$.

If $\nu$ is near the turning point with respect to any of $x=1$, $x_1$, or $x_2$, then we use the corresponding 
estimates from \ref{IK.turn} for those terms.  In the worst case, we pick up an extra factor of $\lambda^\frac13$ from the turning point estimates.   Since $\abs{\nu} \asymp \lambda$ near the turning points, the resulting estimate is $O(\abs{\nu}^{-\frac23})$.
\end{proof}

\bigbreak
We turn next to estimates of the Poisson operator, which is quite straightforward in the physical half-plane.
\begin{prop}\label{E0.prop}
For $\chi \in \cinf_0(0,1)$ and $\re s \ge \nh$,
\[
\mu_k(\chi E_0(s)) \le C e^{c_1 \brak{s} - c_2 k^{1/n}}.
\]
The same estimate holds if $\chi$ is replaced by a radial differential operator with coefficients in $\cinf_0(0,1)$.
\end{prop}
\begin{proof}
Since the cutoff depends only on $x$, the operator $(\chi E_0(s))^* \chi E_0(s)$
is diagonal with respect to the eigenfunctions 
$\phi_\lambda$, with eigenvalues given by
\begin{equation}\label{EE.eig}
\int \abs{\chi(x) b_\lambda(s;x)}^2 \>\frac{dx}{x^{n+1}},
\end{equation}
for $\lambda^2 \in \sigma(\Delta_h)$.
Up to a possible change of ordering, these values correspond with the set of
values of $\mu_k(\chi E_0(s))^2$.  

To analyze the asymptotics, we set $\nu = s - \nh$ and use the conjugation symmetry to restrict our attention to $\im \nu \ge 0$.
From \eqref{bl.def} we have the explicit formula,
\begin{equation}\label{bl.II}
b_\lambda(s;x) = \frac{1}{2\nu} \left(\frac{\lambda}{2}\right)^{\nu} \Gamma(1-\nu) x^\frac{n}2
\left[I_{-\nu}(\lambda x) - \frac{I_{-\nu}(\lambda)}{I_{\nu}(\lambda)}  I_{\nu}(\lambda x)\right],
\end{equation}
for $\lambda>0$.  For $\lambda \le M$, we can deduce from \eqref{Iseries.bnd} that
\begin{equation}\label{bl.bd}
b_\lambda(s;x) \asymp \frac{1}{2\nu} \left[ x^{n-s} - x^s\right],
\end{equation}
uniformly for $x \in \supp \chi$.  For $\lambda = 0$ this formula is exact by \eqref{bl.def}.
Hence for $\re s\ge 0$ and $0\le \lambda \le M$, we have $b_\lambda(s;x) = O(1)$.  

Now assume $\lambda> M$ with $M$ large enough that Proposition~\ref{IK.Ai.prop} applies.  Using \eqref{I.reflect} we can write 
\begin{equation}\label{bl.IK}
b_\lambda(s;x) = \frac{ (\lambda/2)^\nu}{\Gamma(\nu+1)} x^\frac{n}2 \left[ K_{\nu} (\lambda x)
- \frac{K_\nu(\lambda)}{I_\nu(\lambda)} I_\nu(\lambda x) \right].
\end{equation}
Assuming $M$ is sufficiently large, Proposition~\ref{IK.Ai.prop} (along with Corollary \ref{turn.cor} if either $\psi(\nu,\lambda x)$ or $\psi(\nu,\lambda)$ is close to zero) shows that the $K_{\nu} (\lambda x)$ term dominates in 
\eqref{bl.IK}.  The key point is that $x< 1$ and $\re \psi(\nu, \lambda x)$ is a increasing function of $x$.  
Thus for $\lambda> M$ we have
\begin{equation}\label{bl.psi}
\abs{b_\lambda(s;x)} \le C \abs{\frac{(i\lambda/2)^\nu}{\Gamma(\nu+1)}} e^{-\re \psi(\nu,\lambda x)}.
\end{equation}
Applying Stirling's formula then yields
\begin{equation}\label{log.bl}
\log \abs{b_\lambda(s;x)} \le  \re \left[ -\nu \log \left(\frac{2x\nu}{\nu + \sqrt{\nu^2+ \lambda^2 x^2}} \right)
+ \nu - \sqrt{\nu^2+ \lambda^2 x^2} \right] + O(1).
\end{equation}

If $\lambda x \gg \abs{\nu}$ then this estimate reduces to
\[
\log \abs{b_\lambda(s;x)} \le - \lambda x + O\left(\abs{\nu}\log \frac{\lambda x}{\abs{\nu}} \right).
\]
Hence, for $\lambda \ge m\abs{\nu}$ with $m$ sufficiently large, we have
\[
\log \abs{b_\lambda(s;x)} \le - c\lambda.
\]
On the other hand, for $\lambda < m\abs{\nu}$, \eqref{log.bl} clearly shows that
\[
\log\abs{b_\lambda(s;x)} = O(\brak{s}).
\]

The result follows from the formula \eqref{EE.eig} for the eigenvalues of $(\chi E_0(s))^* \chi E_0(s)$ and 
the Weyl asymptotic for the values of $\lambda^2 \in \sigma(\Delta_h)$.

To extend the estimates to include radial derivatives is a straightforward exercise using \eqref{bl.IK} and the identities
\[
\begin{split}
\del_x I_\nu(\lambda x) &= \lambda I_{\nu+1}(\lambda x) + \frac{\nu}{x} I_\nu(\lambda x),\\
\del_x K_\nu(\lambda x) &= - \lambda K_{\nu+1}(\lambda x) + \frac{\nu}{x} K_\nu(\lambda x)
\end{split}
\]
\end{proof}

\bigbreak
The extension of Proposition \ref{E0.prop} to the non-physical plane is complicated by the presence of poles at the resonances.
For this purpose it is most convenient to use the scattering matrix, because the scattering matrix is already diagonalized.

\begin{prop}\label{S0.prop}
For $\re s \ge \nh$, $d(s, n-\calR_0) \ge \brak{s}^{-\beta}$, and 
$d(s, \nh + \bbN_0) \ge \delta$, with $\beta, \delta >0$, we have
\[
\norm{S_0(n-s)} \le e^{C\brak{s} \log \brak{s}}.
\]
\end{prop}
\begin{proof}
Since our Bessel asymptotics are restricted to $\re \nu\ge 0$, 
it is convenient to produce a lower bound of $S_0(s)$ in the region $\re s \ge \nh$ and
then exploit the symmetry $S_0(n-s) = S_0(s)^{-1}$.  Also, by the conjugation symmetry, 
$S_0(\overline{s}) = \overline{S_0(s)}$, we are free to restrict our attention
to the quadrant $\arg \nu \in [0, \tfrac{\pi}2]$.

Consider the eigenvalue
\begin{equation}\label{S0.elt}
[S_0(n-s)]_\lambda = \left(\frac{\lambda}{2}\right)^{-2\nu} \frac{\Gamma(\nu)}{\Gamma(-\nu)}
\frac{I_{\nu}(\lambda)}{I_{-\nu}(\lambda)}.
\end{equation}
For $\lambda \le M$, with $M$ some fixed constant, the asymptotics are quite simple:
\begin{equation}\label{S0.bd}
[S_0(n-s)]_\lambda = - 1 + O(\nu^{-1}),
\end{equation}
for $\abs{\nu}$ sufficiently large. 

Assuming that $\lambda>M$ with $M$ sufficiently large, we can apply Lemmas~\ref{Iratio.lemma} and \ref{Iratio.z} to \eqref{S0.elt}.
For $\arg\nu \in [0,\tfrac{\pi}2]$, $d(\nu, \calR_0 - \nh) \ge \abs{\nu}^{-\beta}$, and  $d(\nu, \bbN_0) \ge \delta$,  we have 
\begin{equation}\label{S0.Ifactor}
 \abs{\frac{I_{\nu}(\lambda)}{I_{-\nu}(\lambda)}} \preceq 
\begin{cases} e^{C \brak{\nu} \log\brak{\nu}} & \abs{\re\psi} \le b, \abs{\psi} \ge c, \\
1 & \text{otherwise},
\end{cases}
\end{equation}
with constants that depend only on $b, c, \beta$, and $\delta$.
Using Stirling's formula and the Euler reflection formula, we find that
\[
\log \frac{\Gamma(\nu)}{\Gamma(-\nu)} = 2\nu \log \nu - (2 + i\pi ) \nu +
O(1),
\]
for $\arg\nu \in [0,\tfrac{\pi}2]$ with $d(\nu, \bbZ) \ge \vep$.  The claimed estimate follows by applying these estimates to \eqref{S0.elt}.
\end{proof}

Using the standard identity 
\[
E_0(n-s) = -E_0(s)S_0(n-s),
\]
we can estimate
\[
\mu_k(\chi E_0(n-s)) \le \mu_k(\chi E_0(s))\norm{S_0(n-s)} .
\]
Hence Propositions \ref{E0.prop} and \ref{S0.prop} together give us the:
\begin{corollary}\label{E0.cor}
For $\chi \in \cinf_0(0,1)$ and $\re s \le \nh$, with $d(s, \calR_0) \ge \brak{s}^{-\beta}$, and 
$d(s, \nh - \bbN_0) \ge \delta$,
\[
\mu_k(\chi E_0(s)) \le C e^{c_1 \brak{s} \log\brak{s} - c_2 k^{1/n}}.
\]
The same estimate holds if $\chi$ is replaced by a radial differential operator with coefficients in $\cinf_0(0,1)$.
\end{corollary}

\section{Resonance order of growth}

For an asymptotically hyperbolic manifold $(X,g)$ with warped-product ends, the model estimates of the previous section lead to growth estimate on the resonance counting function $N_g(t)$.  The basic technique is the Fredholm determinant method of Melrose \cite{Melrose:1983, Melrose:1984}, as adapted to the hyperbolic setting by Guillop\'e-Zworski \cite{GZ:1995a}.  Indeed, the only real difference in our proof from that of \cite{GZ:1995a} lies in the model estimates proven in \S\ref{model.sec}.

Let $R_0(s)$ denote the resolvent for the model end $X_0 = (0,1] \times \Sigma$, as studied in \S\ref{model.sec}.
The resonance set $\calR_0$ was identified explicitly in \eqref{R0.def}, and we let $N_0(t)$ denote the corresponding counting function.  
In Proposition~\ref{modelcount.prop} we will show that 
\begin{equation}\label{N0.bnd}
N_0(t) \sim c \, t^{n+1},
\end{equation}
and compute the constant explicitly.  The main goal of this section is to prove the following:
\begin{prop}\label{Nbnd.prop}
Let $(X,g)$ be a conformally compact manifold with asymptotically hyperbolic warped-product ends.  Then the resonance counting function satisfies
\[
N_g(t) = O((t \log t)^{n+1}).
\]
\end{prop}
\noindent
The bound in Proposition~\ref{Nbnd.prop} is not optimal and will be refined later in \S\ref{sharpbnd.sec}.

Choose smooth cutoff functions $\chi_k \in \cinf_0(X)$, such that $\chi_k = 1$ within $K$ and within $X_0$, 
$\chi_k = 1$ for $r \le k$ and $0$ for $r \ge k+1$.  For some fixed $s_0$ with $\re s_0$ large we define the parametrix
\[
M(s) := \chi_2 R_g(s_0) \chi_1 + (1-\chi_0) R_0(s) (1-\chi_1).
\]
This satisfies
\begin{equation}\label{param}
(\Delta - s(n-s))M(s) = 1 - L(s),
\end{equation}
with the error term
\begin{equation}\label{L.err}
\begin{split}
L(s) &:= - [\Delta, \chi_2] R_g(s_0) \chi_1 + [s(n-s) - s_0(n-s_0)] \chi_2 R_g(s_0) \chi_1 \\
&\qquad + [\Delta, \chi_0] R_0(s) (1-\chi_1).
\end{split}
\end{equation}
Note that $\chi_3 L(s) = L(s)$.  Using this and applying the resolvent to \eqref{param}, we can write
\[
M(s)\chi_3 = R_g(s) \chi_3 (1 - L(s) \chi_3).
\]

With the cutoff included, $L(s) \chi_3$ is a pseudodifferential operator of order $-1$ with compactly supported coefficients. 
Thus, $(L(s) \chi_3)^{n+2}$ is a trace class operator and we can define the Fredholm determinant
\begin{equation}\label{Ds.def}
D(s) := \det \left[1- (L(s) \chi_3)^{n+2} \right].
\end{equation}
From Vodev \cite[Appendix]{Vodev:1994}, we obtain the following:
\begin{lemma}\label{Dres.lemma}
The resonance set $\calR_g$ (counted with multiplicities) is contained within the union of the set of zeros of $D(s)$ and $n+2$ copies of the set $\calR_0 \cup \{\nh\}$.
\end{lemma}
The proof of the Lemma is essentially identical to that of \cite[Cor.~9.3]{Borthwick}.  

\begin{lemma}\label{Ds.lemma}
For $\beta, \delta>0$, suppose that $d(s, \calR_0) \ge \brak{s}^{-\beta}$ and $d(s, \nh - \bbN_0) \ge \delta$.
Then for $\vep>0$ sufficiently small we have
\[
\log \abs{D(s)} \le \begin{cases} C \brak{s}^{n+1} & \text{for } \re s -\nh \ge \vep, \\
C (\brak{s}\log \brak{s})^{n+1} &\text{for } \re s -\nh \le -\vep, \\
C (\brak{s}\log \brak{s})^{n+\frac43} &\text{for }\abs{\re s -\nh} \le \vep. 
\end{cases}
\]
\end{lemma}
\begin{proof}
To estimate the growth of $D(s)$, we separate 
\[
L(s) \chi_3 = T_0 + T_1(s) + T_2(s),
\]
corresponding to the three terms on the right-hand side of \eqref{L.err}.  All terms are compactly supported, and $T_2(s)$ is smoothing and therefore trace class.

By the argument for \cite[Lemma~6.1]{GZ:1995a}, which uses the Weyl inequality for determinants and the Fan inequalities for singular values, we can deduce the bounds,
\[
\abs{D(s)} \le \det \bigl(1 + C_0 \abs{T_0}^{n+2}\bigr)^{k_0}  \det \bigl(1 + C_1 \abs{T_1(s)}^{n+2}\bigr)^{k_1} 
\det\bigl(1 + C_2 \abs{T_2(s)}\bigr)^{k_2},
\]
for some integers $k_j$, $j=0,1,2$.
The first term is just a constant.  To estimate the second term, $T_1(s)$, we note that it is quadratic in $s$, 
\[
T_1(s) = [s(n-s) - s_0(n-s_0)] \chi_2 R_g(s_0) \chi_1.
\]
Thus, since $\chi_2 R_g(s_0) \chi_1$ has order $-2$, we have a bound on singular values,
\[
\mu_k(T_1(s)) \le C\brak{s}^2 k^{-\frac{2}{n+1}}.
\]
Thus
\[
\det \bigl(1 + C_1 \abs{T_1(s)}^{n+2}\bigr) \le \prod_{k=1}^\infty \left( 1 + C\Bigl(\frac{\brak{s}^{n+1}}{k} \Bigr)^{\gamma} \right),
\]
where $2 < \gamma \le 3$.  We can thus estimate
\[
\log \det \bigl(1 + C_1 \abs{T_1(s)}^{n+2}\bigr) \le C \int_{1}^\infty \log \left( 1 + C\Bigl(\frac{\brak{s}^{n+1}}{x} \Bigr)^{\gamma} \right)\>dx 
= O(\brak{s}^{n+1}).\]

Therefore, the proof comes down to a growth estimate on 
\[
\det\bigl(1 + C_2 \abs{T_2(s)}\bigr),
\]
where
\[
T_2(s) := [\Delta, \chi_0] R_0(s) (\chi_3-\chi_1).
\]

From Proposition~\ref{R0est.prop} we can use comparison to eigenvalues of the Laplacian on a compact domain to
deduce a bound,
\[
\mu_k(T_2(s)) \le C \min \left\{ k^{-2} \brak{s}^{2(n+1)}, 1 \right\},
\]
for $\re(s-\nh)\ge \vep$.  We can then apply the Weyl determinant estimate,
\begin{equation}\label{weyl.det}
\log \abs{\det(1+T)} \le \sum_{k=1}^\infty \mu_k(T),
\end{equation}
to deduce 
\begin{equation}\label{detT2a}
\log \det\bigl(1 + C_2 \abs{T_2(s)}\bigr) = O(\brak{s}^{n+1}).
\end{equation}
(See e.g.~the proof of \cite[Lemma~9.12]{Borthwick}.)
Similarly, for $0\le \re(s-\nh)\le \vep$ (assuming $\vep <1/6$), Proposition~\ref{R0est.prop} yields
\begin{equation}\label{detT2b}
\log \det\bigl(1 + C_2 \abs{T_2(s)}\bigr) \le  C\brak{s}^{n+\frac43}.
\end{equation}

To obtain bounds for $\re (s-\nh)\le 0$, we appeal to the identity (see, e.g. \cite[Lemma~9.4]{Borthwick})
\begin{equation}\label{detT2c}
\det(1 + 2 \abs{T_2(n-s)}) \le \det(1 + 2 \abs{T_2(s)})^2 \det(1 + 2 \abs{T_2(s) - T_2(n-s)})^2.
\end{equation}
The first determinant on the right has already been dealt with.  As for the second, we can use the identity
\begin{equation}\label{RR.EE}
R_0(s) - R_0(n-s) =  (2s-n) E_0(s) E_0(n-s)^t,
\end{equation}
to reduce this to a determinant involving
\[
T_2(s) - T_2(n-s) = (2s-n) [\Delta, \chi_0] E_0(s) E_0(n-s)^t (\chi_3-\chi_1).
\] 
By Corollary~\ref{E0.cor}, assuming $d(s, n-\calR_0) \ge \brak{s}^{-\beta}$, and $d(s, \nh + \bbN_0) \ge \delta$,
we find that
\[
\mu_k(T_2(s) - T_2(n-s)) \le C e^{c_1 \brak{s} \log\brak{s} - c_2 k^{1/n}}.
\]
This time we use the Weyl determinant estimate in the form
\[
\abs{\det(1 + T)} \le (1 + \norm{T})^m \exp\left( \sum_{k=m+1}^\infty \mu_k(T) \right),
\]
with $m = (c_1 \brak{s} \log\brak{s}/c_2)^n$.  This yields
\[
\log \det(1 + 2 \abs{T_2(s) - T_2(n-s)}) = O((\brak{s}\log\brak{s})^{n+1}).
\]
By applying this estimate to the second factor in \eqref{detT2c}, and using \eqref{detT2a} and \eqref{detT2b} for the first factor,
we can thereby deduce that \eqref{detT2b} holds for $-\vep \le \re (s-\nh) \le 0$ and 
\eqref{detT2a} holds for $\re (s-\nh)\le -\vep$ with $d(s, \calR_0) \ge \brak{s}^{-\beta}$, and 
$d(s, \nh - \bbN_0) \ge \delta$.  
\end{proof}

\bigbreak
\begin{proof}[Proof of Proposition~\ref{Nbnd.prop}]
To complete the argument, let $\calR_0$ denote the set of resonances of $X_0$.  By the asymptotic \eqref{N0.bnd}, we can form the Weierstrass product,
\[
H_0(s) := \prod_{\zeta\in \calR_0} \left( 1 - \frac{s}{\zeta}\right) \exp\left[\frac{s}{\zeta} + \dots 
+ \frac{1}{n+1} \left(\frac{s}{\zeta} \right)^{n+1} \right].
\]
Lindel\"of's Theorem (see e.g.~\cite[Thm.~2.10.1]{Boas}) shows that the associated entire function 
\[
g_0(s) = H_0(s) H_0(e^{i\pi/(n+1)} s),
\]
is of finite type, so that
\begin{equation}\label{g0.bnd}
\log \abs{g_0(s)} \le C\brak{s}^{n+1}.
\end{equation}

From \eqref{L.err} we can see the poles of $D(s)$ are contained within some finite number of copies of $\calR_0$.  Hence, for some $N>0$, the function $h(s) := g_0(s)^N D(s)$ will be entire.  Using \eqref{g0.bnd} we can apply the bounds from 
Lemma~\ref{Ds.lemma} to $h(s)$.   And since $h(s)$ is entire, we can use the maximum modulus theorem to fill in the missing disks around $\calR_0$ and $\nh - \bbN_0$, and the Phragm\'en-Lindel\"of theorem to extend the stronger bound into the strip at $\re s = \nh$.  The result is that
\[
\log \abs{h(s)} \le C (\brak{s}\log\brak{s})^{n+1},
\]
for all $s \in \bbC$.  Since, by Lemma~\ref{Dres.lemma}, the zero set of $h(s)$ contains $\calR_g$, the counting estimate follows from Jensen's formula.
\end{proof}

\section{Poisson formula}

To establish the Poisson formula for resonances, we need to introduce the relative scattering determinant.  Let $S_g(s)$ and $S_0(s)$ denote the scattering matrices associated to $(X,g)$ and the background manifold $(X_0, g_0)$, respectively.  
By \eqref{param} we have the relation
\begin{equation}\label{MRL}
M(s) = R_g(s) - R_g(s) L(s),
\end{equation}
from which we can derive, by taking boundary limits on the right and left, that
\begin{equation}\label{SSE}
S_0(s) = S_g(s) - (2s-n) E_g(s)^t [\Delta, \chi_0] E_0(s).
\end{equation}
This shows in particular that $S_g(s)S_0(s)^{-1} - 1$ is smoothing and hence trace class on $\Sigma$.  
Thus we can define the relative scattering determinant,
\[
\tau(s) := \det S_0(s)^{-1}S_g(s).
\]

By the order bound of Proposition~\ref{Nbnd.prop}, we can define the Weierstrass product
\[
H_g(s) := \prod_{\zeta\in \calR_g} \left( 1 - \frac{s}{\zeta}\right) \exp\left[\frac{s}{\zeta} + \dots 
+ \frac{1}{n+1} \left(\frac{s}{\zeta} \right)^{n+1} \right],
\]
and we recall that $H_0(s)$ was defined as the corresponding product over $\calR_0$.

\begin{prop}\label{tau.divisor}
The relative scattering determinant admits a Hadamard factorization of the form 
\begin{equation}\label{tau.div.eq}
\tau(s) = e^{q(s)} \frac{H_g(n-s)}{H_g(s)} \frac{H_0(s)}{H_0(n-s)},
\end{equation}
with $q(s)$ a polynomial of order at most $n+1$.
\end{prop}
\begin{proof}
To work out the divisor of $\tau(s)$, we can appeal to the theory developed by Gohberg-Sigal \cite[\S4--5]{GS:1971} 
to deduce that
\[
\res_\zeta \frac{\tau'}{\tau}(s) = \tr \res_\zeta \left[ S_g'(s) S_g(s)^{-1} \right] -  \tr \res_\zeta \left[ S_0'(s) S_0(s)^{-1} \right].
\]
Letting $m_g(\zeta)$ denote the multiplicity of a resonances at $\zeta$, we have the relation 
\begin{equation}\label{nu.mumu}
- \tr \res_\zeta \left[ S_g'(s) S_g(s)^{-1} \right] = m_g(\zeta) - m_g(n-\zeta) + \sum_{k\in \bbN} \Bigl( \mathbbm{1}_{n/2 - k}(\zeta) 
-\mathbbm{1}_{n/2+k}(\zeta) \Bigr) d_k,
\end{equation}
where $d_k$ is the dimension of the kernel of the $k$-th conformal Laplacian on $(\Sigma, h)$.
This result is due to Guillarmou \cite{Gui:2005} (with earlier partial results by \cite{BP:2002, GZ:2003, GZ:1997}, and
with a restriction that was later removed in \cite{GN:2006}).

Since the $d_k$ cancel between the $S_g(s)$ and $S_0(s)$ terms, we obtain
\[
\res_\zeta \frac{\tau'}{\tau}(s) = m_g(n-\zeta) - m_g(\zeta) + m_0(\zeta) - m_0(n-\zeta).
\]
This proves the claimed formula with $q(s)$ an entire function.  It remains to show that $q(s)$ is a polynomial with the claimed order.

Using the parametrix formula \eqref{param} and the fact that $\chi_3 L(s) = L(s)$ we can rewrite the identity \eqref{MRL} as
\[
M(s) = R_g(s) - M(s) (1-L(s)\chi_3)^{-1} L(s).
\]
The corresponding scattering matrix identity is
\[
S_0(s) = S_g(s) - (2s-n)E_0(s)^t (1 - \chi_1) (1-L(s)\chi_3)^{-1} [\Delta, \chi_0] E_0(s).
\]
The relative scattering matrix is thus given by
\begin{equation}\label{tau.form}
\tau(s) = \det \Bigl( 1 - (2s-n) E_0(n-s)^t (1 - \chi_1) (1-L(s)\chi_3)^{-1} [\Delta, \chi_0] E_0(s) \Bigr).
\end{equation}

The $L(s)\chi_3$ term we write as
\[
(1-L(s)\chi_3)^{-1} = \left(1+ L(s)\chi_3 + (L(s)\chi_3)^2\right) \left(1-(L(s)\chi_3)^3\right)^{-1}.
\]
Using Proposition~\ref{R0est.prop}, the identity \eqref{RR.EE}, and Corollary~\ref{E0.cor}, we have
\[
\norm{1+ L(s)\chi_3 + (L(s)\chi_3)^2} = O(e^{C\brak{s}}).
\]
Since $(L(s)\chi_3)^3$ is trace class we can use a resolvent estimate from Gohberg-Krein \cite{GK:1969} to obtain the estimate
\[
\norm{(1-L(s)\chi_3)^{-1}} \le \frac{\det(1+ \abs{L(s)\chi_3}^3)}{D(s)},
\]
where $D(s)$ is the determinant \eqref{Ds.def}.  Lemma~\ref{Ds.lemma} gives the upper bound
\[
\log \abs{D(s)} = O\left((\brak{s}\log\brak{s})^{n+\frac43}\right),
\]
for $d(s, \calR_0) \ge \brak{s}^{-\beta}$ and $d(s, \nh - \bbN_0) \ge \delta$.  We can clearly derive the same estimate for
$\log \det(1+ \abs{L(s)\chi_3}^3)$ by the same argument.  The minimum modulus theorem \cite[8.7.1]{Titchmarsh} shows that
if we assume that $\beta>n+4/3$, then the upper bound for $D(s)$ implies a lower bound
\[
- \log \abs{D(s)} = O(\brak{s}^{n+\frac43+\vep}),
\]
for $\vep>0$, $d(s, \calR_0) \ge \brak{s}^{-\beta}$ and $d(s, \nh - \bbN_0) \ge \delta$.  So our estimate becomes
\begin{equation}\label{L3inv.est}
\norm{(1-L(s)\chi_3)^{-1}} \le e^{C\brak{s}^{m}},
\end{equation}
for $m > n+\frac43$ and $d(s, \calR_0) \ge \brak{s}^{-\beta}$ and $d(s, \nh - \bbN_0) \ge \delta$.

Returning to \eqref{tau.form}, after combining \eqref{L3inv.est} with the singular values estimates for the $E_0(s)$ terms from Corollary~\ref{E0.cor}, we can use the Weyl determinant estimate to deduce that
\[
\log \abs{\tau(s)} = O(\brak{s}^{(n+1)m}),
\]
for $m > n+\frac43$ and $d(s, \calR_0) \ge \brak{s}^{-\beta}$ and $d(s, \nh - \bbN_0) \ge \delta$.
This implies at least that $q(s)$ is polynomial, although with an order bound much higher than claimed.

Once $q(s)$ is known to be polynomial, it suffices to estimate its growth in a sector.  Proposition~\ref{prop2.tau} gives a sharp estimate on
the growth of $\log \tau(s)$ for $\abs{\arg (s-\nh)} \le \tfrac{\pi}2 - \vep$, which shows in particular that $q(s)$ has order at most $n+1$.
\end{proof}

\bigbreak
The Poisson formula follows from Proposition~\ref{tau.divisor}, by essentially the same analysis developed for the surface case by Guillop\'e-Zworski \cite{GZ:1997}.  (See also the versions of this argument in \cite{Borthwick, Borthwick:2008}.)
The crucial step is a Birman-Krein type formula that relates the derivative of the scattering determinant to the $0$-traces of the spectral measures,
\begin{equation}\label{BK.formula}
-\del_s \log \tau(s) = (2s-n) \Bigl( \otr [R_g(s) - R_g(n-s)]   - \otr [R_0(s) - R_0(n-s)]  \Bigr).
\end{equation}
In the present context this follows immediately from a result of Guillarmou \cite[Thm.~3.10]{Gui:2009}, which shows that each $0$-trace on the right is given by the Konsevich-Vishik trace of the logarithmic derivative of the corresponding scattering matrix.  When we take the difference of these two formal traces, we recover the actual trace of the logarithmic derivative of the relative scattering matrix.

The traces on the right in \eqref{BK.formula} are the Fourier transforms of regularized wave traces.  Proposition~\ref{tau.divisor} gives an explicit formula for the left-side and shows that it is a tempered distribution.  Taking the Fourier transform
of \eqref{BK.formula} (as in \cite[Thm~1.2]{Borthwick:2008}, for example), yields the proof of the Poisson formula stated in 
Theorem~\ref{rel.poisson}.

Finally we consider the asymptotics of the scattering phase,
\begin{equation}\label{sigma.def}
\sigma(t) := \frac{i}{2\pi} \log \tau(\nh+it),
\end{equation}
with branches chosen so that $\sigma(t)$ is continuous.  By the properties of the scattering matrix, $\sigma(t)$ is a real-valued
 odd function of $t\in \bbR$.
Using the analysis of the big singularity of the wave traces at $t=0$, developed in the asymptotically hyperbolic case by Joshi-S\'a Barreto \cite{JS:2000}, and the method from Guilop\'e-Zworski \cite[Thm~.1.5]{GZ:1997}, we can derive the:
\begin{corollary}\label{scphase.cor}
Assume $(X,g)$ is asymptotically hyperbolic metric with warped-product ends, with core $K$.
As $t \to +\infty$,
$$
\sigma(t) = W_K t^{n+1} + O(t^n),
$$
where $W_K$ is the Weyl constant
\[
W_K := \frac{(4\pi)^{-\frac{n+1}2}}{\Gamma(\frac{n+3}2)} \vol(K,g).
\]
\end{corollary}

As a final remark, we note that because the Poisson formula Theorem~\ref{rel.poisson} includes the resonances of the background metric $g_0$, it does not lead to a lower bound for resonances along the lines of \cite{GZ:1997} or \cite{Borthwick:2008}.  The technique used in those arguments, based on the big singularity of the wave trace at $t=0$, would produce a lower bound only for the sum $N_g(t) + N_{g_0}(t)$, as in \cite[Cor.~3.2]{BCHP}.  Since we already know that $N_{g_0}(t)$ saturates the resonance bound, by the Weyl law on $\Sigma$, this unfortunately yields no lower bound for $N_g(t)$.

\section{Sharp upper bounds}\label{sharpbnd.sec}

In this section we will refine the crude counting estimate of Proposition~\ref{Nbnd.prop} into the proof of Theorem~\ref{main.thm}.  The first step is to compute the asymptotic constant for the counting function of the model case $(X_0,g)$.  This amounts to counting zeros of Bessel functions, a similar argument to a calculation of Stefanov \cite{Stefanov:2006}.  

Proposition~\ref{tau.divisor} shows how the divisor of the relative scattering determinant $\tau(s)$ is determined by the resonance sets $\calR_g$ and $\calR_0$.   Using a contour integral as in \cite[Prop.~3.2]{Borthwick:2010}, we obtain the formula:
\begin{prop}\label{relcount.prop}
As $a\to \infty$,
\begin{equation*}
\int_0^a \frac{N_g(t)-N_0(t)}{t} \: dt = 2 \int_0^a \frac{\sigma(t)}{t} \: dt
+ \frac{1}{2\pi} \int_{-\tfrac\pi2}^{\tfrac\pi2} \log\abs{\tau(\nh+ae^{i\theta})} \: d\theta +
O(\log a).
\end{equation*}
\end{prop}

The asymptotic for the scattering phase $\sigma(t)$ was given in Corollary~\ref{scphase.cor}.  
Hence for the application of Proposition~\ref{relcount.prop} we must establish the asymptotic for $N_0(t)$ and estimate $\abs{\tau(s)}$ for $\re s \ge \nh$.

\subsection{Asymptotic counting for the model space}
The resonances of the model space were identified explicitly in \eqref{R0.def} as zeros of $I_\nu(\lambda)$, where
$\nu := s-\nh$ and $\lambda^2 \in \sigma(\Delta_h)$.  In this subsection we will use the Bessel function asymptotics from \S\ref{bessel.sec} to work out the constant in the asymptotic that we claimed for the model space counting function $N_0(r)$ in \eqref{N0.bnd}.

Since our Bessel function asymptotics assume that $\re\nu\ge 0$, we will study the zeros through the reflection formula,
\begin{equation}\label{Ireflect2}
I_{-\nu}(\lambda) = I_{\nu}(\lambda) + \frac{2\sin \pi\nu}{\pi} 
K_\nu(\lambda).
\end{equation}
There are two distinct sources of zeros of $I_{-\nu}(\lambda)$.   For $\abs{\nu} \succeq \lambda$, the 
$K_\nu(\lambda)$ term is dominant in \eqref{Ireflect2}.  Thus $I_{-\nu}(\lambda)$ has some zeros which are perturbations of the integer points where $\sin \pi\nu = 0$.  We refer to these as `trivial' zeros, as they are quite easy to count.  Note that because the trivial zeros are perturbations of simple zeros, and the zero set of $I_{-\nu}(\lambda)$ has a conjugation symmetry, the trivial zeros must remain on the real axis.   They can never occur precisely at an integer, however, since $I_{-k}(z) = I_{k}(z)$ for $k\in \bbZ$, which is strictly positive for $z>0$.

The `non-trivial' zeros of $I_{-\nu}(\lambda)$ occur within the red zone shown in Figure~\ref{NuPsiPlot.fig} (and its reflection by conjugation).  Within this zone and away from the real axis, the approximation \eqref{KI.atri} is valid, and the zeros are approximately given by solutions of the equation
\begin{equation}\label{Ai.approx}
\Ai \left(e^{\frac{2\pi i}3}(\tfrac32 \psi)^{2/3}\right) = 0,
\end{equation}
where $\psi = \psi(\nu, \lambda)$, as defined in \eqref{psi.def}.   Recall that $\psi(\nu, \lambda) = \lambda \rho(\alpha)$,  where $\nu = \lambda\alpha$ and
\begin{equation}\label{rhoa.def}
\rho(\alpha) := \sqrt{\alpha^2+1} + \alpha \log \left(\frac{i}{\alpha +
\sqrt{\alpha^2+1}}\right).
\end{equation}
Within the red zone the corresponding values of $\psi$ are close to the positive imaginary axis. Hence we can apply the approximation \eqref{ai.asym2} to \eqref{Ai.approx} to obtain the simpler equation
\begin{equation}\label{cosh.lr}
\cosh\bigl(\lambda\rho(\alpha)-i\tfrac\pi4\bigl) = 0.
\end{equation}
We will count the solutions to \eqref{cosh.lr} that lie on
\[
\gamma := \bigl\{\alpha: \>\re \rho(\alpha) = 0,\,\im\rho(\alpha) \ge 0\bigr\},
\]
and then relate the corresponding counting function to $N_0(r)$.
The curve $\gamma$ is shown in Figure~\ref{gamma.fig}.  Note that the actual resonance lines in Figure~\ref{X0Plot.fig} are well approximated by the reflections of $\gamma$ across the imaginary axis, scaled by the square roots $\lambda$ of the eigenvalues.

\begin{figure}
\begin{center}
\begin{overpic}[scale=.7]{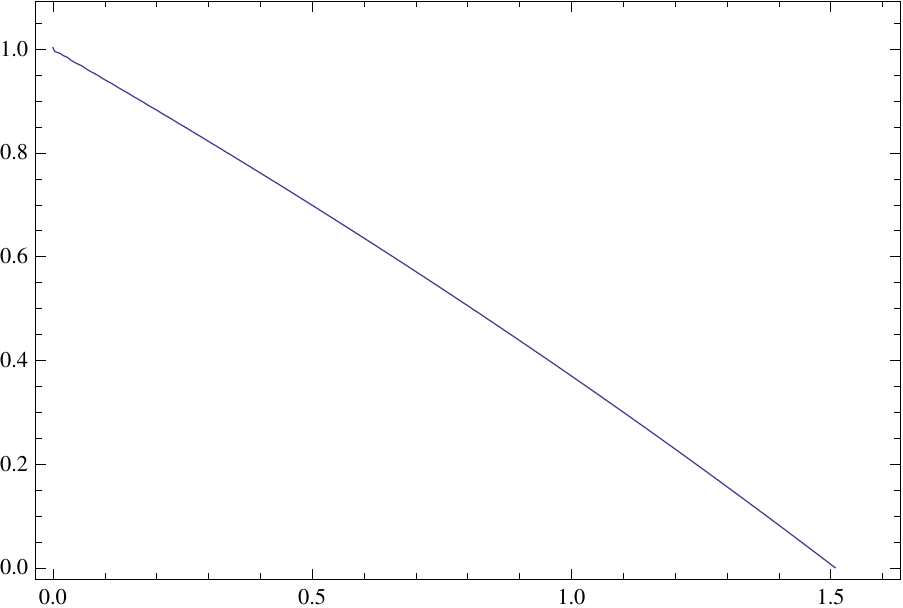}
\put(48,38){$\gamma$}
\end{overpic}
\end{center}
\caption{The curve $\gamma$ containing the solutions of \eqref{cosh.lr}.}\label{gamma.fig}
\end{figure}

Let $W_\Sigma$ denote the Weyl constant for the compact manifold, $(\Sigma, h)$:
\[
W_\Sigma := \frac{\vol(\Sigma,h)}{\Gamma(\frac{n}2+1) (4\pi)^{n/2}},
\]
so that
\[
\#\bigl\{\lambda^2\in\sigma(\Delta_h): \; \lambda\le r\bigr\}\sim W_\Sigma \cdot r^n.
\]

\begin{lemma}\label{auxcount.lemma}
Let $M(r; \theta_1,\theta_2)$ denote the number of zeros of \eqref{cosh.lr} for $\abs{\nu} \le r$, $\arg\nu \in [\theta_1,\theta_2)$, and $\lambda^2 \in \sigma(\Delta_h)\backslash\{0\}$.  For $0 \le \theta_1 <\theta_2 \le \tfrac{\pi}2$ 
this count satisfies the asymptotic
\begin{equation}\label{aux_count}
M(r; \theta_1,\theta_2) = \frac{nW_\Sigma}{(n+1)\pi} \: r^{n+1} \int_{\gamma|_{[\theta_1,\theta_2]}}  \frac{\abs{\rho'(\alpha)}}{\abs{\alpha}^{n+1}}\> \abs{d\alpha} + O(r^{n}),
\end{equation}
where $\gamma$ is parametrized by $\theta = \arg \alpha$.
\end{lemma}
\begin{proof}
For any $\lambda$ the zeros of \eqref{cosh.lr} with $\re\alpha>0$ lie on the curve $\gamma$.  Note that the zeros of \eqref{cosh.lr} with $\re\alpha = 0$ are not included in the count $M(r; \theta_1,\tfrac\pi2)$. 
As an alternative parametrization of $\gamma$, define $\tilde{\gamma}(t)$ implicitly by
\[
\rho(\tilde{\gamma}(t)) = i\pi t,
\]
for $t \in [0, \alpha_0/2]$.  The constant $\alpha_0 \approx 1.509$ is the value of $\alpha$ at which the curve $\gamma$ intersects the real axis.  This determines the range of $t$ since $\rho(\alpha_0) = \frac{i\pi}2 \alpha_0$.

For $0 \le \theta_1 <\theta_2 \le \tfrac{\pi}2$, let $t_1$ and $t_2$ be the corresponding parameters so that 
$\gamma(\theta_j) = \tilde{\gamma}(t_j)$.  For fixed 
$\lambda$, the number of zeros of \eqref{cosh.lr} with $\arg\alpha \in [\theta_1, \theta_2)$ is given exactly by
the number of points in $\lambda(t_2, t_1] \cap (\bbN - \tfrac14)$.  We can thus estimate the number of zeros in this range as \begin{equation}\label{lambda.t.error}
\lambda (t_1-t_2) + O(1),
\end{equation}
where the error term is bounded by $\pm1$.

Now consider the full count, summed over $\lambda$.  The number of $\lambda$'s for which $\gamma$ intersects $\{\abs{\alpha} \le r/\lambda\}$ is $O(r^n)$ by the Weyl law, so that by applying \eqref{lambda.t.error} for each $\lambda$ and summing the errors we obtain
\begin{equation}\label{M.approx}
M(r; \theta_1, \theta_2) =  \tilde{M}(r; \theta_1, \theta_2) + O(r^n),
\end{equation}
where
\[
\tilde{M}(r; \theta_1, \theta_2) := \sum_\lambda \lambda \, \ell\Big(\tilde{\gamma}^{-1} \Big[ \gamma|_{[\theta_1, \theta_2]}\cap \{\abs\alpha\le\tfrac{r}{\lambda}\} \Big]\Big).
\]

Fix some $\theta$ and small $\Delta\theta$, and define $t$ and $\Delta t$ by $\tilde{\gamma}(t)=\gamma(\theta)$ and $\tilde{\gamma}(t-\Delta t)=\gamma(\theta+\Delta\theta)$.   Then we can estimate
\[
\abs{\tilde{M}(r; \theta, \theta+\Delta\theta)  - \sum_{\lambda\abs{\gamma(\theta)} \le r} \lambda \Delta t \>}
\le \sum_{r/\abs{\gamma}_{\rm max}\le\lambda\le r/\abs{\gamma}_{\rm min}} \lambda\,\Delta t
,\]
where the extrema of $\abs{\gamma}$ are taken over the sector $\arg\alpha\in[\theta,\theta+\Delta\theta]$.
Since $\abs{\gamma}_{\rm max} - \abs{\gamma}_{\rm min}$ is bounded by $c\Delta\theta$, the Weyl law implies that the number of 
terms in the sum on the right hand side of the above inequality is $O(r^n\Delta\theta)$. 
Hence we can write
\[
\tilde{M}(r; \theta, \theta+\Delta\theta) = \sum_{\lambda \le r/\abs{\gamma(\theta)}} \lambda \Delta t \quad 
+ O(r^{n+1} \Delta t \Delta\theta).
\]
By the Weyl law,
\[
\sum_{\lambda \le r/\abs{\gamma(\theta)}} \lambda =  \frac{nW_\Sigma}{n+1} \left( \frac{r}{\abs{\gamma(\theta)}}\right)^{n+1} + O(r^{n}).
\]
Since $\abs{\Delta t} \le c \abs{\Delta\theta}$, we conclude
\begin{equation}\label{tM.Weyl}
\tilde{M}(r; \theta, \theta+\Delta\theta) = \frac{nW_\Sigma}{n+1} \left( \frac{r}{\abs{\gamma(\theta)}}\right)^{n+1} \Delta t 
+ O(r^n \Delta t)
+ O(r^{n+1}(\Delta\theta)^2)
\end{equation}

We now pass to an integral over $\theta$ and derive an estimate for $\tilde{M}(r; \theta_1,\theta_2)$ from \eqref{tM.Weyl}.
Then from \eqref{M.approx} we obtain
\[
M(r; \theta_1,\theta_2) = \frac{nW_\Sigma}{n+1} r^{n+1} \int_{\gamma|_{[\theta_1,\theta_2]}}  \frac{dt}{\abs{\gamma}^{n+1}}
+ O(r^{n}).
\]
The final step is to note that $(\rho\circ\tilde{\gamma})'(t) = i\pi$, 
so that the change of variables from $t$ to arclength is accounted for by introducing a factor of $\abs{\rho'}/\pi$.
\end{proof}

\bigbreak
\begin{prop}\label{modelcount.prop}
The resonance counting function $N_0$ for the model space $X_0$ satisfies the asymptotic
\begin{equation}\label{model_count}
N_0(r) = \Bigg[ \frac{2n W_\Sigma}{(n+1)\pi} \int_{\gamma}
\frac{\abs{\rho'(\alpha)}}{\abs{\alpha}^{n+1}} \: d\abs\alpha 
+ \frac{W_\Sigma}{n+1} \alpha_0^{-n} \Bigg] r^{n+1} + O\big(r^{n+\frac13}\big),
\end{equation}
where $W_\Sigma$ is the Weyl constant for $(\Sigma,h)$, $\gamma=\{\alpha: \>\re \rho(\alpha) = 0,\,\im\rho(\alpha) \ge 0\}$, and $\alpha_0$ is the real solution to $\re\rho(\alpha)=0$.
\end{prop}
\begin{proof}
From Lemma \ref{Iratio.lemma} we know that the non-trivial zeros of $I_{-\nu}(\lambda)$ in $\im\nu\ge0$ and for $\lambda$ sufficiently large are contained in the region
\[S_{\lambda,b}=\{\arg\nu\in[0,\tfrac\pi2]:\; \im\psi\ge0,\abs{\re\psi}\le b\}.\]
From Corollary \ref{IK.cor} and \eqref{I.reflect} we deduce that for $\psi$ and $\lambda$ sufficiently large we have
\begin{equation}\label{II.approx}
 \frac{I_{-\nu}(\lambda)}{I_\nu(\lambda)}-(1+ie^{-2\psi})
=ie^{-2\psi}\bigr[O(\psi^{-1})+O(\lambda^{-1})-e^{2i\pi\nu}\bigl],
\end{equation}
for $\nu\asymp\lambda$. Note that the zeros of the function
\[1+ie^{-2\psi} = 2 e^{-\psi + i\frac{\pi}4} \cosh\bigl(\psi-i\tfrac\pi4\bigl),\]
in \eqref{II.approx} correspond precisely to the solutions of \eqref{cosh.lr}.

We will now prepare to apply Rouch\'e's theorem to the functions on the left hand side of \eqref{II.approx}.
Let $\nu_{\lambda,m}$ denote the solution of \eqref{cosh.lr} for which
\[\psi(\nu_{\lambda,m})=i\pi (m - \tfrac14),\]
with $m \in \bbN \cap [\frac14,\frac14 + \frac{\lambda\alpha_0}{2}]$. Define $\gamma_{\lambda,m}$
to be the contour obtained from the lines
\[
\{\re\psi=b\},\quad\{\re\psi=-b\},\quad\{\im\psi=\pi(m-\tfrac{3}{4})\},\quad\{\im\psi=\pi(m+\tfrac{1}{4})\},
\]
in the $\nu$-plane. Then each $\nu_{\lambda,m}$ lies within $\gamma_{\lambda,m}$, and further we have that on $\gamma_{\lambda,m}$
\[\abs{e^{-2\psi}}\le\beta\abs{1+ie^{-2\psi}},\]
where the constant $\beta$ depends only on $b$.

In order to control the right hand side of \eqref{II.approx} we define for some $\sigma, \tau>0$ the region
\[\Gamma_{\sigma,\tau} := \left\{\nu:\>\im\nu\ge\tau,\> \re\nu \ge \sigma (\im \nu)^{1/3}\right\}.\]
Recalling from the proof of Corollary \ref{turn.cor}
that for some small $\delta$
\[ \psi\asymp\lambda^{-\tfrac12}(\nu-i\lambda)^{\tfrac32} \qquad \text{for~} \abs{\nu-i\lambda}<\delta\lambda
,\]
we see that by letting both $\sigma$ and $\tau$ be large enough, \eqref{II.approx} yields
\[ \abs{\frac{I_{-\nu}(\lambda)}{I_\nu(\lambda)}-(1+ie^{-2\psi})}
< \frac1\beta \abs{e^{-2\psi}},\]
on $S_{\lambda,b}\cap\Gamma_{\sigma,\tau}$ for $\lambda$ sufficiently large.

Rouch\'e's theorem now implies that for $\lambda$ sufficiently large, $I_{-\nu}(\lambda)$ has exactly one zero within every $\gamma_{\lambda,m}$ that is contained in $\Gamma_{\sigma,\tau}$. Also, there are no other zeros since the contours $\gamma_{\lambda,m}$ cover the regions $S_{\lambda,b}$. The diameters of the $\gamma_{\lambda,m}$ are $O(\lambda^{1/3})$ with a constant that depends only on $b$. Consequently,
\[ 
\#\left\{ \zeta\in\calR_0:\> \abs{\zeta-\nh} \le r,\>-\zeta \in \Gamma_{\sigma,\tau} \right\} \le M(r+cr^{1/3};0,\tfrac\pi2) + O(r^n).
\]
After noting that $M(r-cr^{1/3};\theta_1,\theta_2)$ provides a lower bound for all $0<\theta_1<\theta_2<\tfrac\pi2$, we conclude
\begin{equation}\label{Gcount}
\#\left\{ \zeta\in\calR_0:\> \abs{\zeta-\nh} \le r,\>-\zeta \in \Gamma_{\sigma,\tau} \right\} = M(r;0,\tfrac\pi2) + O\big(r^{n+\tfrac13}\big).
\end{equation}

For each $\lambda$ that is sufficiently large, the number of zeros of $I_{-\nu}(\lambda)$ in $S_{\lambda,b}-\Gamma_{\sigma,\tau}$ is uniformly bounded by a constant depending only on $\sigma$ and on $\tau$. This shows that the number of non-trivial zeros of $I_{-\nu}(\lambda)$ with 
$\abs{\nu} \le r$, $\arg\nu\in[0,\tfrac\pi2]$, and $\nu \notin\Gamma_{\sigma,\tau}$, is $O(r^n)$. 
Therefore
\[
\#\left\{ \zeta\in\calR_0:\> \abs{\zeta-\nh} \le r,\>\im \zeta \ne 0 \right\} = 2M(r;0,\tfrac\pi2) + O\big(r^{n+\tfrac13}\big).
\]

It remains to show that the contribution of the trivial zeros to the counting function is given by the second term in the constant claimed in \eqref{model_count}.   There is a positive constant $c$ such that for $\lambda$ sufficiently large and real $\nu \ge \lambda (1-\vep)\alpha_0$ we have
\[
I_{-\nu}(\lambda) = \frac{2}{\pi} K_\nu(\lambda) \Bigl(\sin \pi \nu + O\big(e^{-2c(\nu-\lambda\alpha_0)}\big)\Bigr).
\]
Hence
\[
\#\left\{ \zeta\in\calR_0:\> \abs{\zeta-\nh} \le r,\>\im \zeta = 0 \right\}  = \sum_{\lambda: \,\lambda \alpha_0 \le r} (r - \lambda \alpha_0) + O(1).
\]
This expression is easily estimated using the Weyl law for $\{\lambda\}$.  We conclude that
\begin{equation}\label{N0.zero}
\#\left\{ \zeta\in\calR_0:\> \abs{\zeta-\nh} \le r,\>\im \zeta = 0 \right\} = \frac{W_\Sigma}{n+1} \alpha_0^{-n} r^{n+1} + O(r^n).
\end{equation}
\end{proof}

From the proof of Proposition~\ref{modelcount.prop} we observe that in the model case we have a resonance-free region with boundary given by a cube-root: for some small $\sigma$,
\[
\calR_0 \cap \left\{s\in\bbC:\> \re(s-\nh) \ge -\sigma \abs{\im s}^{1/3}\right\} = \emptyset.
\]

\subsection{Estimate of the scattering determinant}
The goal of this section is to find an upper bound for $\log\abs{\tau(s)}$ with $s$ in a
sufficiently big subset of $\{\re s > \nh, \im s \ge 0 \}$. 

For some small $\eta>0$, let $x_j = 1 - \eta j$ for $j=0,1,2,3$.  We choose cutoff functions $\chi_j \in \cinf((0,1])$ so that
$\chi_j(x) = 1$ for $x \ge x_j$ and $\chi_j(x) = 0$ for $x \le x_{j+1}$. 
With the model Poisson operator $E_0(s)$ defined as in \S\ref{poisson.op.sec}, we can express the relative scattering determinant as
\[
\tau(s) = \det\left(1+(2s-n)E_0(s)^t [\Delta_0,\chi_2] R(s) [\Delta_0,\chi_1] E_0(n-s)\right).
\]
Since the derivations of this identity and of Lemma \ref{prop1.tau} follow \cite[Lemmas~4.1~and~5.2]{Borthwick:2010} closely, we omit the details.  
\begin{lemma}\label{prop1.tau}
For $\re s \ge \nh$ with $d(s(n-s), \sigma(\Delta_g))\ge \vep$, the relative scattering determinant can be estimated by
\begin{equation}
\label{est.tau}
\log\abs{\tau(s)} \le \sum_{\lambda^2 \in \sigma(\Delta_h)} \log \bigl( 1+ C \kappa_\lambda(s)\bigr),
\end{equation}
where \[
\kappa_\lambda^2(s)= |2s-n|^2 \int_{x_2}^{x_1} x^{-(n+1)} |b_{\lambda}(n-s,x)|^2 \:
dx \int_{x_3}^{x_2} x^{-(n+1)} |b_{\lambda}(s,x)|^2 \: dx,
\]
with the coefficients $b_\lambda(s;x)$ as defined in \eqref{bl.def}, and where the constant $C$ depends only on $\eta$ and $\vep$.  
\end{lemma}

Using the identity 
\[E_0(s)=-E_0(n-s)S_0(s),\]
we find
\begin{equation}\label{kappa.k}
\kappa_\lambda(s) \le \abs{\frac{I_\nu(\lambda)}{I_{-\nu}(\lambda)}} 
\abs{\,\nu\, \Bigl(\frac{\lambda}{2}\Bigr)^{-2\nu} \frac{\Gamma(\nu)}{\Gamma(-\nu)}}
 \int_{x_3}^{x_1} x^{-(n+1)} \abs{b_{\lambda}(s,x)}^2 \, dx
,\end{equation}
for $k>0$. Define the following set of radii $a$, for
which the corresponding circles stay away from the zeros of the scattering matrix in the sense
of Proposition \ref{S0.prop}.
\[
\Lambda:=\bigl\{a\in \bbR_+ :\>\min_\theta d(ae^{i\theta},\nh-\calR_0)\ge \brak{a}^{-\beta}, 
d(a,\bbN_0) \ge \delta \bigr\},
\]
where $\beta > n+1$ and $\delta>0$.
Then, for $\abs\nu\in \Lambda$, we have control of $I_\nu/I_{-\nu}(\lambda)$ by \eqref{S0.Ifactor}.
The requirement for Lemma \ref{prop1.tau}, that $d(s(n-s),\sigma(\Delta_g))\ge \vep$, will be satisfied if
$\abs\theta\le\frac\pi2-\vep a^{-2}$ for $\nu=ae^{i\theta}$ with $a$ sufficiently large. Under these two restrictions we obtain:
\begin{prop}\label{prop2.tau}
 For $a \in \Lambda$ we have
 \[
  \log\abs{\tau(\nh+ae^{i\theta})}\le B(\theta) a^{n+1} + o(a^{n+1})
 ,\]
uniformly for $\abs\theta\le\frac\pi2-\vep a^{-2}$, with
\[
B(\theta) := 2nW_\Sigma\int_0^\infty \frac{[-\re \rho(xe^{i\abs\theta})]_+}{x^{n+2}} \: dx
,\]
where $W_\Sigma$ is the Weyl constant for $\Delta_h$ and $[\cdot]_+$ denotes the positive part.
\end{prop}
\begin{proof}
We again use the conjugation symmetry to restrict our attention to $\im\nu\ge0$. 
For $\lambda > M$, where $M$ is the constant from Lemma \ref{Iratio.lemma}, 
\eqref{bl.psi} and \eqref{kappa.k} give
\begin{equation}\label{kappa.k2}
\kappa_\lambda(s) \le
C \, e^{-2\lambda \re\rho(\alpha,x_3)} g_\lambda(\nu),
\end{equation}
where $\nu = s-\tfrac{n}2$ and $\alpha := \frac{\nu}{\lambda}$ as always, and
\[
g_\lambda(\nu) := \abs{\frac{I_{\nu}(\lambda)}{I_{-\nu}(\lambda)}}.
\]

Given $\nu=ae^{i\theta}$, we split the sum (\ref{est.tau}) (minus the $\lambda=0$ term) according to the sign
of $\re\rho(\tfrac{\nu}{\lambda},x_3)$.
The sum over $\lambda$ with $\re\rho(\alpha,x_3)< 0$ is finite
and we further divide it into contributions from the Poisson kernel and from the scattering matrix.  
Since the terms in the sum \eqref{est.tau} with $\lambda \le M$ are $O(a)$ 
(for $\lambda > 0$ by \eqref{S0.elt}, \eqref{S0.bd} and \eqref{bl.bd}), we have
\[
\log\abs{\tau(s)}\le \Sigma_L + \Sigma_P + \Sigma_S + O(a),
\]
where
\begin{equation*}
 \begin{split}
\Sigma_L & := \sum_{\lambda>0:\,\re\rho(\alpha,x_3)\ge 0} \log \bigl( 1+Ce^{-2\lambda
\re\rho(\alpha,x_3)}g_\lambda(\nu)
\bigr), \\
\Sigma_P & :=  \sum_{\lambda>0:\,\re\rho(\alpha,x_3)< 0} 2\lambda\re[-\rho(\alpha,x_3)], \\
\Sigma_S & :=  \sum_{\lambda>0:\,\re\rho(\alpha,x_3)< 0} \log(1+Cg_\lambda(\nu)).
 \end{split}
\end{equation*}

Let us now index the spectrum $\{\lambda>0\}$ by $\lambda_k$, $k\in\bbN$.
Define the constant $\omega=W_\Sigma^{-1/n}$, so that $\lambda_k\sim\omega k^{1/n}$ by Weyl's law. 
Recall that $\re \rho(\alpha,x)$ is monotonically increasing in $x$,
and define the function $A(\theta)$ and the constants $q,Q>0$ by
\[
\re\rho(A(\theta)e^{i\theta},x_3)=0, \qquad
q\omega k^{1/n} \le \lambda_k\le Q\omega k^{1/n}
.\]

First we show that $\Sigma_L$ contributes only of lower order. For $a$ sufficiently large, the factors $g_\lambda(\nu)$ in $\Sigma_L$ are bounded by Lemma \ref{Iratio.lemma}, and therefore we have
\[
\Sigma_L \le C \sum \exp\bigl(-2\lambda_k[\re \rho(\alpha,x_3)]_+\bigr)
\]
for $a$ sufficiently large,
where the sum is for $k$ from $\lfloor(\tfrac{a}{Q\omega A(\theta)})^n\rfloor$ to $\infty$. 
The monotonicity in $x$ of $\re\rho(\alpha,x)$ yields
\[
\Sigma_L \le  C \sum \exp\bigl(-2[\re\rho(ae^{i\theta}, q\omega k^{1/n} x_3)]_+\bigr)
.\]
Switching to an integral over $k$ and substituting $x=\tfrac{a}{\omega k^{1/n}}$, we find
\[
\Sigma_L \le Ca^n \int_0^{QA(\theta)} \frac1{x^{n+1}}\exp\Bigl(-\frac{2a}{x}
[\re\rho\bigl(xe^{i\theta},q x_3\bigr)]_+\Bigr) \: dx + O(1)
.\]
Since $\rho(0,qx_3)>0$, the integral exists, and the fact that the integrand is decreasing in $a$ shows that $\Sigma_L=O(a^n)$.

For an estimate of $\Sigma_P$, define numbers $\mu_k$ with 
$\lambda_k=(1+\mu_k)\omega k^{1/n}$ for all k. Then
\[
2 \lambda_k [-\re\rho(\alpha,x_3)]_+ = 2 \omega k^{1/n} 
 \bigl[-\re\rho\bigl(\tfrac{a}{\omega k^{1/n}}e^{i\theta},(1+\mu_k)x_3\bigr)\bigr]_+
.\]
The number of $\mu_k$ with absolute value greater than $\eta/x_3$ is finite and independent of $a$. The corresponding
terms in the sum $\Sigma_P$ are $O(a)$. For all other $k$ we have $(1+\mu_k)x_3\ge x_3-\eta=:x_4$ and hence, by
monotonicity of $\re\rho$ and letting $x=\tfrac{a}{\omega k^{1/n}}$ as above,
\begin{equation*}
\Sigma_P \le \sum 2\omega k^{1/n} [-\re \rho(xe^{i\theta},x_4)]_+ + O(a), 
\end{equation*}
where the sum is for $k$ from $1$ to $\lceil(\tfrac{a}{q\omega A(\theta)})^n\rceil$.
Switching to the corresponding integral over $x$, we find
\[
\Sigma_P \le \frac{2n}{\omega^n} \int_0^\infty \frac{[-\re \rho(xe^{i\theta},x_4)]_+}{x^{n+2}} \: dx 
\cdot a^{n+1} + O(a).
\]
With a simple change of variables we can scale the $x_4$ out of the integral, yielding
\begin{equation}\label{Sigma.P}
\Sigma_P \le x_4^{-n} B(\theta) a^{n+1} + O(a).
\end{equation}

The number of terms in the sum $\Sigma_S$ is,
as for $\Sigma_P$, less than $\lceil(\tfrac{a}{q\omega A(\theta)})^n\rceil$. Most of them are bounded by Lemma
\ref{Iratio.lemma}, and Lemma \ref{Iratio.z} states that the remaining terms,
those for which $\nu$ is in the 'red zone' of Figure~\ref{NuPsiPlot.fig},
are $O(a\log a)$ for $a\in\Lambda$. More precisely:
\begin{equation}\label{Sigma.S}
\Sigma_S \le O(a^n) + K(\nu) \, c \, a \log a \qquad \qquad (a\in\Lambda),
\end{equation}
where, with $b$ being the constant from Lemma \ref{Iratio.lemma}, 
\begin{equation}\label{Knu}
K(\nu)=\#\bigl\{k>0:\, \im\psi(\nu,\lambda_k)\ge 0, \, \abs{\re\psi(\nu,\lambda_k)}\le b\}
.\end{equation}
Now suppose that for fixed $\nu$ we have $\mu_1,\mu_2$ with
\[\re\psi(\nu,\mu_1)=-b, \qquad \re\psi(\nu,\mu_2)=+b,\]
where the condition on $\mu_1$ is replaced by $\im\psi(\nu,\mu_1)=0$ for small $\arg\nu$. We observe that $K(\nu)$ is given by the number of eigenvalues $\lambda_k$ between $\mu_1$ and $\mu_2$. Since the width of the region
\[\{\nu: \abs{\re\psi(\nu,\mu)}\le b,\, \im\psi(\nu,\mu)\ge0\}\]
in the $\nu$-plane is $O(\mu^{1/3})$ uniformly in $\theta=\arg\nu$, we can estimate
\[K(\nu)\le W_\Sigma\Biggl(\biggl(\frac{a}{\abs{\gamma(\theta)}}+ca^{1/3}\biggr)^n
-\biggl(\frac{a}{\abs{\gamma(\theta)}}-ca^{1/3}\biggr)^n\Biggr)+O(a^{n-1}).\]
This shows $K(\nu)=O(a^{n-2/3})$, and by \eqref{Sigma.S} we obtain $\Sigma_S=O(a^{n+1/3} \log a)$.

We conclude
\[
\log\abs{\tau(\nh+ae^{i\theta})} \le x_4^{-n} B(\theta) a^{n+1}+C_{\vep,\eta}\, a^{n+1/3} \log a 
,\]
where the constant might blow up as $\eta\rightarrow0$. This gives
\[
\limsup_{a\rightarrow \infty} \left[ \frac{\log\abs{\tau(\nh+ae^{i\theta})}}{a^{n+1}}-B(\theta) \right]
\le (x_4^{-n} - 1) B(\theta)
,\]
which, by letting $x_4\rightarrow 1$, completes the proof.
\end{proof}

\subsection{Completing the sharp estimate}
\begin{proof}[Proof of Theorem \ref{main.thm}]
With the asymptotics of the scattering phase, as stated in Corollary \ref{scphase.cor}, the relative counting formula from Proposition \ref{relcount.prop} becomes 
\[
(n+1)\int_0^a \frac{N_g(t)-N_0(t)}{t}\:dt= 2W_Ka^{n+1}
+ \frac{n+1}{2\pi} \int_{-\tfrac\pi2}^{\tfrac\pi2} \log\abs{\tau(\nh+ae^{i\theta})} \: d\theta + o(a^{n+1})
.\]
We can then apply the asymptotic for $N_0(t)$ from Proposition \ref{modelcount.prop} and the scattering determinant estimate from Proposition \ref{prop2.tau}.  Comparing the result to \eqref{constants} shows that for Theorem \ref{main.thm} it remains to show that the contribution of
\[\int_{\tfrac\pi2-\vep a^{-2}\le\abs\theta\le\tfrac\pi2} \log\abs{\tau(\nh+ae^{i\theta})} \: d\theta,\]
is of lower order.  If we assume $a\in\Lambda$ then by the Hadamard factorization \eqref{tau.div.eq} of $\tau$ and the Minimum Modulus Theorem \cite[Thm.~8.71]{Titchmarsh}, we have the estimate
\[
\abs{\tau(\nh+ae^{i\theta})}\le C_\epsilon \exp\big(a^{n+1+\epsilon}\big),
\]
for any $\epsilon>0$, provided $\beta > n+1$ in the definition of $\Lambda$.    This implies 
\[
\int_{\tfrac\pi2-\vep a^{-2}\le\abs\theta\le\tfrac\pi2} \log\abs{\tau(\nh+ae^{i\theta})} \: d\theta = O(a^{n-1+\epsilon}),
\]
which suffices to complete the proof.
\end{proof}

\end{document}